\RequirePackage{fix-cm}
\documentclass[smallextended]{svjour3}       
\smartqed  
\usepackage{mathrsfs}
\usepackage{eepic}
\usepackage{amsfonts}
\usepackage{verbatim}
\usepackage{ragged2e}
\usepackage[numbers,sort&compress]{natbib}
\usepackage{arydshln}

\usepackage{color}
\usepackage{booktabs}
\usepackage{threeparttable}
\usepackage{amssymb}
\usepackage{bbding}
\usepackage{booktabs}
\usepackage{graphicx}
\usepackage{dsfont}
\usepackage{amsmath}
\usepackage{float}
\usepackage{subfigure}
\usepackage{multirow}
\usepackage{multicol}
\usepackage[figureright]{rotating}
\usepackage{lscape}
\usepackage{epstopdf}

\begin{document}
	
	\title{The augmented weak sharpness of solution sets in equilibrium problems
		\thanks{This research was supported by the National Natural Science Foundations of China (12371305), the Natural Science Foundation of Shandong Province (ZR2021MA066, ZR2023MA020), the Natural Sciences and Engineering Research Council of Canada (RGPIN-2018-05687), and a centennial fund of University of Alberta.}}
	
	\titlerunning{The augmented weak sharpness of solution sets in equilibrium problems}        

	\author{Ruyu Wang\and  Wenling Zhao\and Daojin Song \and Yaozhong Hu  
	}
	
	
	\institute{Ruyu Wang\at
		School of Mathematics and Statistics, Shandong University of Technology \\
		\email{13121204020@stumail.sdut.edu.cn}          \\
	     \and
	    Wenling Zhao (Corresponding author)\at
	     School of Mathematics and Statistics, Shandong University of Technology \\
	     \email{wenlingzhao@sdut.edu.cn}          \\
	     \and
	     Daojin Song\at
	     School of Mathematics and Statistics, Shandong University of Technology \\
	     \email{djs@sdut.edu.cn}          \\
	     \and
	Yaozhong Hu \at
	Faculty of Science - Mathematics \& Statistical Sciences, University of Alberta\\
	\email{yaozhong@ualberta.ca}
	}
	
	\date{Received: date / Accepted: date}

	\maketitle
	
	\begin{abstract}
	This study delves into equilibrium problems, focusing on the identification of finite solutions for feasible solution sequences. We introduce an innovative extension of the weak sharp minimum concept from convex programming to equilibrium problems, coining this as weak sharpness for solution sets. Recognizing situations where the solution set may not exhibit weak sharpness, we propose an augmented mapping approach to mitigate this limitation. The core of our research is the formulation of augmented weak sharpness for the solution set, a comprehensive concept that encapsulates both weak sharpness and strong non-degeneracy within feasible solution sequences. Crucially, we identify a necessary and sufficient condition for the finite termination of these sequences under the premise of augmented weak sharpness for the solution set in equilibrium problems. This condition significantly broadens the scope of existing literature, which often assumes the solution set to be weakly sharp or strongly non-degenerate, especially in the context of mathematical programming and variational inequality problems. Our findings not only shed light on the termination conditions in equilibrium problems but also introduce a less stringent sufficient condition for the finite termination of various optimization algorithms. This research, therefore, makes a substantial contribution to the field by enhancing our understanding of termination conditions in equilibrium problems and expanding the applicability of established theories to a wider range of optimization scenarios.
		\keywords{ Equilibrium problems\and Feasible solution sequence\and Regular normal cone\and Weak sharpness\and Augmented weak sharpness\and Finite termination}
	\end{abstract}

\section{Introduction}\label{sec1}

In this paper, we explore the equilibrium problem denoted as $EP(\phi, S)$:
\begin{eqnarray*}
	\text{Find}\quad \bar{x}\in S\quad \text{such~that}\quad \phi(\bar{x},y)\geq 0,\quad \forall y\in S,
\end{eqnarray*}
where $S\subset \mathbb{R}^{n}$ is a closed convex set, and
$\phi:\mathbb{R}^{n}\times \mathbb{R}^{n}\rightarrow \mathbb{R}$ is a function such that
\begin{eqnarray}\label{eq1.1}
	\phi(x,x)=0,\quad \forall x\in S.
\end{eqnarray}
Let 
$$\bar{S}=\{x\in S\mid \phi(x,y)\geq0,\quad\forall y \in
S\}\neq\emptyset$$
be the solution set of $EP(\phi,S)$ and
$$\tilde{S}=\{x\in S\mid \exists u\in\partial_{y}\phi(x,x),\quad\text{such~that}\quad\left\langle u,y-x\right\rangle \geq0,~\forall y\in S\}$$
is the stationary points set of $EP(\phi,S)$, where $\partial_{y}\phi(x,x)$ is generalized sub-differential at $x$ of $\phi(x,\cdot)$ ([\cite{rockafellar1998variational}. Definition 8.3]), or sub-differential for short.

The relation between $\bar{S}$ and $\tilde{S}$ is discussed, where it is noted that while the general relation 
\begin{eqnarray}\label{eq1.2}
	\bar{S}\subseteq \tilde{S}
\end{eqnarray}
does not always hold, under certain conditions, such as lower semi-continuity of $\phi(x,\cdot)$ and the satisfaction of $(BCQ)$ constraint qualification in $\bar{S}$, this inclusion is established. This condition is particularly valid when $\phi(x,\cdot)$ is locally Lipschitzian or convex on $\mathbb{R}^n$.

The model $EP(\phi,S)$ serves as a unified model encompassing various optimization problems, including mathematical programming, variational inequality, Nash equilibrium, and saddle point problems, for example, \cite{bigi2013existence,blum1994optimization,combettes2005equilibrium,fan1972minimax}. The study extends to vector optimization problems, demonstrating the versatility of $EP(\phi, S)$. Previous research efforts have expanded the model to include generalized quasi-variational inequality problems. The concept of equilibrium problem plays a central role in various applied sciences, such as physics, economics, engineering, transportation, sociology, chemistry, biology, and other fields \cite{flaam1996equilibrium,konnov2007equilibrium,moudafi2007finite}. The theory of gap functions, developed in the variational inequalities, is extended to a general equilibrium problem in \cite{mastroeni2003gap}. Van et al. \cite{van2020linear} provide sufficient conditions and characterizations for linearly conditioned bifunction associated with an equilibrium problem.

Regarding the finite termination of algorithms for $EP(\phi,S)$, particularly in mathematical programming and variational inequality problems, existing research has predominantly concentrated on concepts such as weak sharp minimum and strong non-degeneracy of the solution set. Pioneering studies by scholars like Rockafellar \cite{rockafellar1976monotone}, Polyak \cite{polyak1987introduction}, Ferris \cite{ferris1991finite}, and others have laid down conditions for finite termination utilizing specific algorithms. Nonetheless, the reliance on algorithmic frameworks has highlighted the necessity for more expansive research into conditions that ensure finite termination, irrespective of the algorithms employed. Early significant contributions in this area were made by Burke and Mor$\acute{e}$ \cite{burke1988identification}, who established necessary and sufficient conditions for the finite termination of feasible solution sequences in smooth programming problems that converge to strongly non-degenerate points. Subsequently, Burke and Ferris \cite{burke1991characterization} broaden these findings to encompass differentiable convex programming. In a further extension, Marcotte and Zhu \cite{marcotte1998weak} generalize these principles to continuous variational inequality problems characterized by pseudo-monotonicity$^+$. Al-Homidan et al. \cite{al2017weak} consider weak sharp solutions for the generalized variational inequality problem, in which the underlying mapping is set-valued. Huang et al. \cite{huang2018weak} give several characterizations of the weak sharpness in terms of the primal gap function associated with the mixed variational inequality. Nguyen \cite{nguyen2021weak} presents the concept of weak sharpness in variational inequality problems, specifically within the context of Hadamard spaces. Subsequently, numerous researchers have explored the concept of weak sharpness of the solution set, particularly focusing on its implications for the finite convergence of diverse algorithms applied to variational inequality problems. This topic has been extensively investigated in various studies, as detailed in references \cite{al2017finite,liu2018weakly,wu2018characterizations,wu2004weak}, among others.

This paper introduces and elaborates on the concepts of weak sharpness and strong non-degeneracy of the solution set within the framework of $EP(\phi,S)$. To tackle scenarios where these characteristics are not present, we propose an augmented mapping on the solution set. This leads to the definition of augmented weak sharpness of the solution set for feasible solution sequences. This novel concept not only generalizes weak sharpness and non-degeneracy but is also employed in establishing necessary and sufficient conditions for the finite termination of feasible solution sequences under the premise of augmented weak sharpness.

The remainder of the paper is organized as follows. Section 2 provides preliminaries and discusses several special cases of $EP(\phi,S)$. In Section 3, the notion of augmented weak sharpness is introduced for the solution set of $EP(\phi,S)$ under general conditions. Section 4 presents examples illustrating situations where weak sharpness or non-degeneracy is not satisfied, but augmented weak sharpness holds. Section 5 establishes the finite identification of feasible solution sequences under the condition of augmented weak sharpness and presents consequences, generalizing results from conditions of weak sharpness or strong non-degeneracy. We make a conclusion in Section 6. 

\section{Preliminary}\label{sec2}
In this section, we introduce fundamental concepts and specific cases relevant to $EP(\phi,S)$, laying the groundwork for subsequent discussions.

Consider an infinite sequence $N \subseteq \{1, 2, \cdots\}$ and sets $C^{k} \subset \mathbb{R}^{n}$ for $k=1, 2, \cdots$. Define the upper limit and lower limit of the sequence of sets as follows:
\begin{eqnarray*}
	\limsup\limits_{k\rightarrow \infty }C^{k} &=&\{x \in \mathbb{R}^{n}\mid 
	\exists N,~x^{k}\in C^k,\quad\text{such~that}~\lim\limits_{k\in N,k\rightarrow \infty }x^{k} = x\},\\
	\liminf_{k\rightarrow \infty }C^{k} &=&\{x \in \mathbb{R}^{n}\mid 
	\exists x^{k}\in C^{k},\quad\text{such~that}~\lim\limits_{k\rightarrow \infty} x_{k}=x\}.
\end{eqnarray*}
Thus we have
\begin{eqnarray*}
	\liminf_{k\rightarrow \infty }C^{k} \subseteq\limsup\limits_{k\rightarrow \infty }C^{k}.
\end{eqnarray*}

Let $C\subset \mathbb{R}^{n},\bar{x}\in C$. The tangent cone for $C$ at $\bar{x}$ is defined as
\begin{eqnarray*}
	T_{C}(\bar{x})=\{d\in \mathbb{R}^{n}\mid \exists x^{k}\in
	C,x^{k}\rightarrow \bar{x}, \tau^{k} \downarrow 0, (k\rightarrow
	\infty),\quad\text{with}\quad\lim\limits_{k\rightarrow
		\infty}(x^{k}-\bar{x})/\tau^{k}=d\}.
\end{eqnarray*}
The regular normal cone for $C$ at $\bar{x}$ is defined as
\begin{eqnarray*}
	\hat{N}_{C}(\bar{x})=\{d\in \mathbb{R}^{n} \mid \left\langle d,x-\bar{x}\right\rangle \leq o(\|x-\bar{x}\|),x\in C\}.
\end{eqnarray*}
In general meaning, the normal cone for $C$ at $\bar{x}$ is
defined as $N_{C}(\bar{x})=\limsup\limits_{x^k\in C,x^k\rightarrow
	\bar{x}}\hat{N}_C(x^k)$. 
The polar cone of $C$ is defined as 
$C^{\circ}=\{y\in \mathbb{R}^{n}\mid \left\langle y,x\right\rangle \leq 0,\quad\forall x\in C\}$. 
By [\cite{rockafellar1998variational}, Proposition 6.5], we have
$T_{C}(x)^{\circ}=\hat{N}_C(x)$. When $C$ is convex, by [\cite{rockafellar1998variational}, Theorem 6.9], we have
\begin{eqnarray*}
	N_{C}(\bar{x})=\hat{N}_C(\bar{x})=\{d\mid \left\langle d,x-\bar{x}\right\rangle \leq 0,\quad\forall x\in C\}.
\end{eqnarray*}
The projection of a point $x\in \mathbb{R}^n$ over a closed set $C$
is defined as
\begin{eqnarray*}
	P_C(x)=\arg\min\limits_{y\in C}\|y-x\|,
\end{eqnarray*}
while the distance from $x\in \mathbb{R}^n$ to $C$ is given by
$\text{dist}(x,C)=\inf\limits_{y\in C}\|y-x\|$. If $C$ is a closed set, then $\text{dist}(x,C)=\|P_{C}(x)-x\|.$

Suppose the subdifferential of $\psi(\cdot)$ at $x\in C$ satisfies
$\partial\psi(x)\neq\emptyset$. Then the projected sub-differential
of $\psi(\cdot)$ at $x$ is
\begin{eqnarray*}
	P_{T_{C}(x)}(-\partial\psi(x))=\{P_{T_{C}(x)}(-u)\mid u\in \partial\psi(x)\}.
\end{eqnarray*}
If $\psi(\cdot)$ is continuously differentiable in a neighborhood of
point $x\in C$, by [\cite{rockafellar1998variational}, Exercise 8.8], we have $\partial\psi(x)=
\{\nabla\psi(x)\}$, which means the projected sub-differential is
the projected gradient $P_{T_{C}(x)}(-\nabla\psi(x)).$

We call that a sequence $\{x^{k}\}\subset \mathbb{R}^{n}$ terminates finitely
to $C$ if there exists $k_0$ such that $x^{k}\in C$ for all $k\geq
k_0$. In $EP(\phi,S)$, we call $\phi$ is monotonic on $S\times S$, if
$\phi(x,y)+\phi(y,x)\leq 0$, for $\forall(x,y)\in S\times S$. A function $\phi$ is said to be pseudo-monotone on $S\times S$, if $\phi(x,y)\geq 0 \Longrightarrow \phi(y,x)\leq 0$, for $\forall(x,y)\in S\times S$.

Below we give several special cases of $EP(\phi,S)$ (see \cite{blum1994optimization}).
\begin{example}\label{Example 2.1} \quad Let
	\begin{eqnarray}\label{eq2.1}
		\phi(x,y)=f(y)-f(x),\quad (x,y)\in \mathbb{R}^{n}\times \mathbb{R}^{n},
	\end{eqnarray}
	where $f:\mathbb{R}^{n} \rightarrow \mathbb{R}$. Obviously, $\phi$ satisfies \eqref{eq1.1}. Then $EP(\phi,S)$ is the following mathematical programming problem:
	\begin{eqnarray*}
		(MP)\qquad\qquad\text{Find}\quad \bar{x}\in S\subseteq \mathbb{R}^{n},\quad \text{such~ that}\quad f(\bar{x})\leq f(y),\quad \forall y\in S.
	\end{eqnarray*}
	By \eqref{eq2.1}, $\phi$ is monotonic over $\mathbb{R}^{n}\times \mathbb{R}^{n}$.
\end{example}

\begin{example}\label{Example 2.2} \quad Let
	\begin{eqnarray}\label{eq2.2}
		\phi(x,y)=\left\langle F(x),y-x\right\rangle ,\quad (x,y)\in S\times S,\quad S\subseteq \mathbb{R}^{n},
	\end{eqnarray}
	where $F:S\rightarrow \mathbb{R}^{n}$. Obviously, $\phi $ satisfies \eqref{eq1.1}. Then $EP(\phi,S)$ is the following variational inequality problem:
	\begin{eqnarray*}
		(VIP)\qquad\qquad\text{Find}\quad \bar{x}\in S,\quad \text{such~ that}\quad \left\langle F(\bar{x}),y-\bar{x}\right\rangle \geq 0,\quad \forall y\in S.
	\end{eqnarray*}
	By \eqref{eq2.2}, we know $\phi$ is monotonic on $S\times S\Longleftrightarrow F$ is monotonic on $S$, and $\phi$ is pseudo-monotonic on $S\times S\Longleftrightarrow F$ is pseudo-monotonic over $S$.
\end{example}

\begin{example}\label{Example 2.3} \quad Suppose $S=S_{1}\times S_{2}$, $S_{1}\subseteq\mathbb{R}^{n_{1}}$, $S_{2}\subseteq \mathbb{R}^{n_{2}}$, $(n_{1}+n_{2}=n)$ are all nonempty closed convex sets, $x=(x_{1},x_{2})\in \mathbb{R}^{n_1}\times \mathbb{R}^{n_2}$, $y=(y_{1},y_{2})\in \mathbb{R}^{n_1}\times \mathbb{R}^{n_2}.$ Let
	\begin{eqnarray}\label{eq2.3}
		\phi(x,y)=\varphi(y_{1},x_{2})-\varphi(x_{1},y_{2}),\quad (x,y)\in \mathbb{R}^{n}\times \mathbb{R}^{n},
	\end{eqnarray}
	where $\varphi:\mathbb{R}^{n}\rightarrow \mathbb{R}$. Obviously, $\phi$ satisfies
	\eqref{eq1.1}. Then $EP(\phi,S)$ is the following global saddle point
	problem:
	\begin{eqnarray*}
		(SPP)\qquad\text{Find}\quad \bar{x}=(\bar{x}_{1},\bar{x}_{2})\in S_{1}\times S_{2},\quad \text{such~ that}&\quad \varphi(\bar{x}_{1},y_{2})\leq \varphi(\bar{x}_{1},\bar{x}_{2})\leq \varphi(y_{1},\bar{x}_{2}),\\
		&\quad \forall (y_{1},y_{2})\in S_{1}\times S_{2}.
	\end{eqnarray*}
	By \eqref{eq2.3}, $\phi$ is monotonic over $\mathbb{R}^{n}\times \mathbb{R}^{n}$.
\end{example}

\begin{example}\label{Example 2.4}
	Let $I=\{1,2,\cdots n\}$, $S=\Pi_{i\in I}S_{i}$, $S_{i}\subseteq \mathbb{R}~(i\in I)$ be nonempty closed convex sets. For $x=(x_{1},x_{2},\cdots ,x_{n})\in \mathbb{R}^{n}$, define
	\begin{eqnarray*}
		x^{i}=(x_{1}\cdots x_{i-1},x_{i+1},\cdots ,x_{n})
	\end{eqnarray*}
	with the obvious modifications for the cases $i=1$ and $i=n$.
\end{example}

Suppose
\begin{eqnarray}\label{eq2.4}
	\phi(x,y)=\Sigma_{i\in I}(f_{i}(x^{i},y_{i})-f_{i}(x)),\quad (x,y)\in \mathbb{R}^{n}\times \mathbb{R}^{n},
\end{eqnarray}
where $f_{i}:\mathbb{R}^{n}\rightarrow \mathbb{R}$. Obviously, $\phi$ satisfies \eqref{eq1.1}. Then $EP(\phi,S)$ is the following Nash equilibrium problem:
\begin{eqnarray*}
	(NEP)\qquad \text{Find}\quad \bar{x}\in S,\quad \text{such~ that}\quad f_{i}(\bar{x})\leq f_{i}(\bar{x}^{i},y_{i}),\quad \forall y_{i}\in S_{i},\quad\text{for~all}\quad i\in I.
\end{eqnarray*}

\section{The augmented weak sharpness in the equilibrium problem}

In this section, we present the notions of weak sharpness and strong
non-degeneracy for the solution set $\bar{S}\subset S$ of
$EP(\phi,S)$. Furthermore, in order to provide a more relaxed
conditions on the finite identification of a feasible solution
sequence, we introduce an augmented mapping over the solution set
$\bar{S}$, and establish the concept of augmented weak sharpness for
the solution set $\bar{S}$ on feasible solution sequence. Under
several different cases and very general assumptions, we prove that
this new concept is a generalization of weak sharpness and strong
non-degeneracy.

First, we give the concept of weak sharp minimum in mathematical
programming (see \cite{burke1993weak,ferris1988weak})

\begin{definition}\label{Definition 3.1}
	In mathematical programming $(MP)$, the solution set $\bar{S}\subset S$ is weak sharp minimal, if there exists a constant $\alpha> 0$, such that for $\forall x\in \bar{S}$, we have
	\begin{eqnarray}\label{eq3.1}
		f(y)-f(x)\geq \alpha\cdot \text{dist}(y,\bar{S}),\quad \forall y\in S.
	\end{eqnarray}
	The constant $\alpha$ and the set $\bar{S}$ are called the modulus
	and domain of sharpness for $f$ over $S$, respectively. Clearly,
	$\bar{S}$ is a set of global minima for $f$ over $S$.
\end{definition}

If $(MP)$ is a non-smooth convex programming, then $\bar{S}$ is a weak
sharp minimal set with the module $\alpha$, if and only if
\begin{eqnarray}\label{eq3.2}
	\alpha B\subset \partial f(x)+[T_{S}(x)\cap N_{\bar{S}}(x)]^{\circ},\quad \forall x\in \bar{S},
\end{eqnarray}
where $B$ is a unit ball (see [\cite{burke1993weak}, Theorem 2.6, c]).

If $(MP)$ is a smooth convex programming, then $\bar{S}$ is a weak
sharp minimal set with the module $\alpha$, if and only if
\begin{eqnarray}\label{eq3.3}
	-\nabla f (\bar{x})\in \text{int}\bigcap_{x\in\bar{S}}[T_{S}(x)\cap
	N_{\bar{S}}(x)]^{\circ},\quad \forall \bar{x}\in \bar{S},
\end{eqnarray}
(see [\cite{burke1993weak}, Corollary 2.7, c)]). Here, in the case that $f$ is smooth, \eqref{eq3.3} and \eqref{eq3.2} are equivalent, as $\nabla f(\cdot)$ is a constant vector on $\bar{S}$ (see [\cite{burke1991characterization}, Corollary 6]).

To generalize the characteristics of the solution set into
variational inequalities and the smooth non-convex programming
problems, some literature (\cite{marcotte1998weak,wang2005convergence,wang2013global,xiu2005finite,zhou2012new}) utilize \eqref{eq3.3} to define the weak sharpness of the solution set in these two kinds of problems. Now we use \eqref{eq3.2} to define the weak sharpness of the solution set $\bar{S}$ of $EP(\phi,S)$.

\begin{definition}\label{Definition 3.2}
	In the $EP(\phi,S)$, for $\forall x\in S$, $\partial_y\phi(x,x)\neq\emptyset$, the solution set $\bar{S}\subset S$ is said to be a weak sharp set with the module $\alpha$, if there exists a constant $\alpha> 0$, such that
	\begin{eqnarray}\label{eq3.4}
		\alpha B\subset \partial_{y}\phi(x,x)+[T_{S}(x)\cap \hat{N}_{\bar{S}}(x)]^{\circ},\quad \forall x\in \bar{S}.
	\end{eqnarray}
\end{definition}

\begin{remark}\label{Remark 3.1}
	In \eqref{eq3.4}, we have used $\hat{N}_{\bar{S}}(\cdot)$
	instead of $N_{\bar S}(\cdot)$, since in general, $\bar{S}$ is not
	necessarily convex. When $\bar{S}$ is non-convex, according to [\cite{rockafellar1998variational}, Proposition 6.5], $\hat{N}_{\bar{S}}(\cdot)$ is a closed convex cone, while $N_{\bar S}(\cdot)$ is only a closed cone, and it holds that
	\begin{eqnarray*}
		\hat{N}_{\bar{S}}(\cdot)\subseteq N_{\bar S}(\cdot).
	\end{eqnarray*}
	When $\bar{S}$ is convex, according to [\cite{rockafellar1998variational}, Theorem 6.9], the formula
	above holds inequality. So we can obtain
	\begin{eqnarray*}
		[T_{S}(x)\cap N_{\bar{S}}(x)]^{\circ}\subseteq [T_{S}(x)\cap \hat{N}_{\bar{S}}(x)]^{\circ}.
	\end{eqnarray*}
	Therefore, $\hat{N}_{S}(\cdot)$ makes the conditions of Definition \ref{Definition 3.2} more relaxed.
\end{remark}

Next, we introduce the notion of strong non-degeneracy in $EP(\phi,S)$.

\begin{definition}\label{Definition 3.3}
	In the $EP(\phi,S)$, suppose $\phi(x,\cdot)$ is differentiable at each $x$ in $S$. We call $\bar{S}\subset S$ is strongly non-degenerate, if
	\begin{eqnarray}\label{eq3.5}
		-\nabla_{y}\phi(\bar{x},\bar{x})\in \text{int}N_{S}(\bar{x}),\quad \forall {\bar{x}\in \bar{S}},
	\end{eqnarray}
	and $\bar{x}$ is said to be a strongly non-degenerate point.
\end{definition}

Now we give the main notion in this paper.

\begin{definition}\label{Definition 3.4}
	In the $EP(\phi,S)$, suppose $\bar{S}\subset S$ is a closed set. For any $x\in S$, there is $\partial_{y}\phi(x,x)\neq \emptyset$, and $\{x^{k}\}\subset S.$ We call $\bar{S}$ is augmented weak sharp with respect to $\{x^{k}\}$. For an infinite sequence $K=\{k\mid x^{k}\notin \bar{S}\}$, there exists an augmented mapping (set-valued mapping) $H:\bar{S}\rightarrow 2^{\mathbb{R}^{n}}$ such that the following hold:
	
	$(a)$ there exists a constant $\alpha>0$, such that
	\begin{eqnarray*}
		\alpha B\subset H(z)+[T_{S}(z)\bigcap\hat{N}_{\bar{S}}(z)]^{\circ},\quad \forall z\in \bar{S},
	\end{eqnarray*}
	
	$(b)$ for $\forall u^{k}\in \partial_{y}\phi(x^{k},x^{k})$ and $\forall v^{k}\in H(P_{\bar{S}}(x^{k}))$, it holds that
	\begin{eqnarray*}
		\limsup\limits_{k\in K ,k\rightarrow \infty}\psi_{k}=\frac{1}{\|x^{k}-P_{\bar{S}}(x^{k})\|} \langle u^{k}-v^{k},x^{k}-P_{\bar{S}}(x^{k})\rangle\geq0.
	\end{eqnarray*}
\end{definition}

Now we will discuss the inclusion relation between two concepts,
augmented weak sharpness of the solution set $\bar{S}$ and the weak
sharpness as well as the strong non-degeneracy in several cases.

\subsection{The non-smooth case}

\begin{proposition}\label{Proposition 3.1}
	In the $EP(\phi,S)$, suppose $\bar{S}\subset S$ is a closed set, $\partial_{y}\phi(x,x)\neq \emptyset$ for any $x\in S$, and $\partial_{y}\phi(x,x)$ is monotonic over $S$. If $\bar{S}$ is weakly sharp, then for every $\{x^{k}\}\subset S$, $\bar{S}$ is augmented weakly sharp.
\end{proposition}
\begin{proof}
	Let $K=\{k\mid x^{k}\notin \bar{S}\}$ be an infinite
	sequence. Set
	\begin{eqnarray*}
		H(z)=\partial_y \phi(z,z),\quad \forall z\in \bar{S}.
	\end{eqnarray*}
	By \eqref{eq3.4} we know that $(a)$ in Definition \ref{Definition 3.4} holds, and (b) also holds by the monotonicity of $\partial_y\phi(x,x)$.
\end{proof}

The following proposition provides a sufficient condition for the
monotonicity of $\partial_y\phi(x,x)$.

\begin{proposition}\label{Proposition 3.2}
	In the $EP(\phi,S)$, suppose that $\phi$ satisfies the following conditions:
	
	$(i)$ For $\forall x\in S$, $\phi(x,\cdot)$ is convex function over
	$\mathbb{R}^{n}$,
	
	$(ii)$ The function $\phi$ is monotonic (pseudo-monotonic) over $S\times S$.
	
	Then $\partial_y \phi(x,x)$ is monotonic (pseudo-monotonic) over $S$.
\end{proposition}
\begin{proof}By $(i)$ we have $\partial_{y}\phi(x,x)\neq \emptyset$,
	for $\forall x\in S$. Now set $u_{x}\in
	\partial_{y}\phi(x,x),u_{z}\in
	\partial_{y}\phi(z,z)$. Suppose $\phi$ is monotonic over $S\times S$. Then we have
	\begin{eqnarray*}
		0 &\geq & \phi(x,z)+\phi(z,x)\\
		&=& \phi(x,z)-\phi(x,x) +\phi(z,x)-\phi(z,z)\quad [by\quad \eqref{eq1.1}]\\
		&\geq&\left\langle u_{x},z-x\right\rangle +\left\langle u_{z},x-z\right\rangle \qquad [by\quad (i)]\\
		&=&\left\langle u_{x}-u_{z},z-x\right\rangle,
	\end{eqnarray*}
	i.e., $\partial_{y}\phi(x,x)$ is
	monotonic over $S$. Suppose $\phi$ is pseudo-monotonic over $S\times
	S$. If
	$$\left\langle u_{x},z-x\right\rangle \geq0,$$
	then we can get
	\begin{eqnarray*}
		\phi(x,z)&=&\phi(x,z)-\phi(x,x)\\
		&\geq & \left\langle u_{x},z-x\right\rangle \geq0 .
	\end{eqnarray*}
	Therefore, according to the pseudo-monotonicity of $\phi$, it holds that
	$$0\geq\phi(z,x)=\phi(z,x) -\phi(z,z)\geq \left\langle u_{z},x-z\right\rangle ,$$
	i.e., $\partial_{y}\phi(x,x)$ is pseudo-monotonic over $S$.
\end{proof}

\begin{remark}\label{Remark 3.2}
	The following two examples show that the two assumptions in Proposition \ref{Proposition 3.2} are only sufficient and not necessary conditions for that $\partial_{y}\phi(x,x)$ is monotonic over $S$.
\end{remark}

\begin{example}\label{Example 3.1}
	In the $EP(\phi,S)$, let $\phi(x,y)=e^{x-y}-e^{y-x}$ and $S=[0,1].$ Then
	$\nabla_{y}\phi(x,x)$ is monotonic over $S$, while $\phi(x,\cdot)$
	is non-convex over $S$, for $\forall x\in [0,1)$.
\end{example}
\begin{example}\label{Example 3.2}
	In $EP(\phi,S)$, let $\phi(x,y)=e^{y^{2}-x^{2}}-1$ and $S=R$. Then $\nabla_{y}\phi(x,x)$ is monotonic over $S$, while $\phi$ is not non-monotonic over $S\times S$.
\end{example}
We note that [\cite{burke1991characterization}, Theorem 5] gave the characteristic description for the solution set of a non-smooth convex programming. The result in [\cite{burke1991characterization}, Theorem 5] not only contributes to the understanding of the nature of the solution set of convex programming, but also plays an important role in the analysis of the weak sharp minimality of the solution set in convex programming. So next this result is generalized to the solution set of $EP(\phi,S)$, and further apply it to the analysis on the weak sharpness of the solution set of $EP(\phi,S)$.

\begin{theorem}\label{Theorem 3.1}
	In $EP(\phi,S)$, suppose that $\phi$ satisfies the following conditions:
	
	$(i)$ For $\forall x\in S$, $\phi(x,\cdot)$ is a convex function
	over $\mathbb{R}^{n}$.
	
	$(ii)$ For $\forall x, z\in \bar{S}$, there is
	$\phi(x,y)=\phi(z,y),\quad \forall y\in \mathbb{R}^{n}$.\\
	Furthermore, suppose $\bar{x}\in \bar{S}$, and $A$ is a convex set
	satisfying $\bar{S}\subseteq A \subseteq S$. Then we have
	\begin{eqnarray}
		\bar{S}&=&\{x\in S\mid \partial_{y} \phi(x,x)\cap(-N_{S}(x))=\partial_{y} \phi(\bar{x},\bar{x})\cap(-N_S(\bar{x}))\}\label{eq3.6}\\
		&=&\{x\in A\mid \partial_{y} \phi(x,x)\cap(-N_{S}(x))=
		\partial_{y} \phi(\bar{x},\bar{x})\cap(-N_{S}(\bar{x}))\}.\label{eq3.7} 
	\end{eqnarray}
\end{theorem}
\begin{proof}Let
	$$\hat{S}=\{x\in S\mid \partial_{y}
	\phi(x,x)\cap(-N_{S}(x))= \partial_{y} \phi(\bar{x},\bar{x})\cap(-N_{S}(\bar{x}))\}.$$
	Next we only need to prove that \eqref{eq3.6} holds, i.e.,
	$\bar{S}=\hat{S}$. Since $A$ satisfies $\bar{S}\subseteq A\subseteq
	S$, with $S$ replaced with $A$ in \eqref{eq3.6}, \eqref{eq3.7} yields immediately. We first prove $\bar{S}\subseteq\hat{S}$. Let $\hat{x}\in \bar{S}$. According to $(i)$ we have $\partial_{y} \phi(\hat{x},\hat{x})
	\bigcap(-N_{S}(\hat{x}))\neq \emptyset$. Take
	$\hat{u}\in\partial_{y} \phi(\hat{x},\hat{x})
	\bigcap(-N_{S}(\hat{x}))$, we can immediately have $\hat{u}\in
	\partial_{y} \phi(\hat{x},\hat{x})$, and
	\begin{eqnarray}\label{eq3.8}
		\left\langle \hat{u},x-\hat{x}\right\rangle \geq0,\quad \forall x\in S.
	\end{eqnarray}
	Since $\hat{x}, \bar{x}\in \bar{S}$, according to $(ii)$ and
	$\phi(x,x)=0$, we get
	$\phi(\hat{x},\bar{x})=\phi(\bar{x},\bar{x})=0$. Thus by a
	combination of $(i)$ and \eqref{eq3.8} we obtain that
	$$0=\phi(\hat{x},\bar{x})-\phi(\hat{x},\hat{x})\geq\left\langle \hat{u},\bar{x}-\hat{x}\right\rangle \geq0.$$
	Therefore,
	\begin{eqnarray}\label{eq3.9}
		\left\langle \hat{u},\bar{x}-\hat{x}\right\rangle =0.
	\end{eqnarray}
	For
	$\forall y\in \mathbb{R}^{n}$, According to $(i)$, $(ii)$ and \eqref{eq3.9}, we
	immediately obtain that
	\begin{eqnarray*}
		\phi(\bar{x},y)-\phi(\bar{x},\bar{x})&=&
		\phi(\hat{x},y)-\phi(\hat{x},\hat{x})\\
		&\geq&\left\langle \hat{u},y-\hat{x}\right\rangle \\
		&=&\left\langle \hat{u},y-\bar{x}\right\rangle +\left\langle \hat{u},\bar{x}-\hat{x}\right\rangle \\
		&=&\left\langle \hat{u},y-\bar{x}\right\rangle .
	\end{eqnarray*}
	So $\hat{u}\in\partial_{y}\phi(\bar{x},\bar{x})$, and according to \eqref{eq3.8}, \eqref{eq3.9}, we obtain that
	\begin{eqnarray*}
		\left\langle \hat{u},x-\bar{x}\right\rangle &=&\left\langle \hat{u},x-\hat{x}\right\rangle +\left\langle \hat{u},\hat
		{x}-\bar{x}\right\rangle \\
		&=&\left\langle \hat{u},x-\hat{x}\right\rangle \geq0
	\end{eqnarray*}
	holds for $\forall x\in S$, i.e., $\hat{u}\in -N_{S}(\bar{x})$. Thus
	$$\partial_{y} \phi(\hat{x},\hat{x})\cap(-N_{S}(\hat{x}))
	\subseteq \partial_{y} \phi(\bar{x},\bar{x})\cap(-N_{S}(\bar{x})).$$
	Since $\bar{x},\hat{x}$ are selected from $\bar{S}$ arbitrarily,
	the inverse inclusion relation of the above relationship also
	holds. Therefore, $\bar{S}\subseteq \hat{S}.$
	
	Now we prove $\hat{S}\subseteq \bar{S}$. Assume $\hat{x}\in\hat{S}.$
	Since $\phi(\bar{x},\bar{x})\bigcap(-N_{S}(\bar{x}))\neq\emptyset$,
	then $\partial_{y}
	\phi(\hat{x},\hat{x})\bigcap$\\$(-N_{S}(\hat{x}))\neq\emptyset.$
	Taking $\hat{u}\in\partial_{y}
	\phi(\hat{x},\hat{x})\bigcap(-N_{S}(\hat{x}))$, for $\forall y\in
	S$ we have
	\begin{eqnarray*}
		\phi(\hat{x},y)&=&\phi(\hat{x},y)-\phi(\hat{x},\hat{x})\\
		&\geq&\left\langle \hat{u},y-\hat{x}\right\rangle \geq0,
	\end{eqnarray*}
	i.e., $\hat{x}\in\bar{S}$, $\hat{S}\subseteq\bar{S}.$
\end{proof}

\begin{remark}\label{Remark 3.3}
	For the convex programming problem $(MP)$, the specific case of equilibrium problem $EP(\phi,S)$, by Example \ref{Example 2.1}, $(i),~(ii)$ obviously hold. But in fact, $(i),(ii)$ hold not only for the convex programming in $EP(\phi,S)$. For example, when $\phi(x,y)=(1-x^{2})(y-x)$, $S=[-1,1]$, $(i)$, $(ii)$ also hold obviously.
\end{remark}

Let
$$G:=\bigcap_{X\in\bar{S}}[T_{S}(x)\bigcap\hat{N}_{\bar{S}}(x)]^{\circ},$$
we have the following proposition.

\begin{proposition}\label{Proposition 3.3}
	Under the assumptions of Theorem \ref{Theorem 3.1}, and furthermore suppose that $\bar{S}\subset S$ is closed, $\partial_{y} \phi(x, x)$ is monotonic over $S$. If
	\begin{eqnarray}\label{eq3.10}
		-\partial_{y} \phi(x,x )\bigcap (-N_{S}(x))\subset \text{int} G \quad\text{for}\quad \forall x\in \bar{S},
	\end{eqnarray}
	then $\bar{S}$ is augmented weakly sharp for all $\{x^{k}\}\subset S$.
\end{proposition}
\begin{proof}
	First, by the assumption $(i)$ in Theorem \ref{Theorem 3.1} and [\cite{tyrrell1970convex}, Theorem 23.4], we get that $\partial_{y} \phi(x,x)$ is a nonempty compact set for $\forall x\in \mathbb{R}^{n}$. Furthermore, by Theorem \ref{Theorem 3.1}, we obtain that $\partial_{y} \phi(x,x)\bigcap (-N_{S}(x))$ is a nonempty compact constant set over $\bar{S}$. So by \eqref{eq3.10}, there exists a constant $\alpha>0$ such that for $\forall x\in \bar{S}$,
	$$\alpha B-\partial_{y} \phi(x,x)\cap(-N_{S}(x))
	\subset[T_{S}(x)\cap[\hat{N}_{\bar{S}}(x)]^{\circ}.$$
	
	From the formula above, for $\forall x\in \bar{S}$, we get that
	\begin{eqnarray*}
		\alpha B &\subset & \partial_{y} \phi(x,x
		)\cap(-N_{S}(x))+[T_{S}(x)\cap[\hat{N}_{\bar{S}}(x)]^{\circ}\\
		&\subset&\partial_{y}
		\phi(x,x)+[T_{S}(x)\cap[\hat{N}_{\bar{S}}(x)]^{\circ}.
	\end{eqnarray*}
	Therefore, $\bar{S}$ is a weak sharp set by Definition \ref{Definition 3.2}.
	According to the monotonicity of $\partial_{y} \phi(x,x)$ and
	Proposition \ref{Proposition 3.1}, the proof is complete.
\end{proof}

\subsection{The smooth case}

In this subsection, we assume that $\phi(x,\cdot)$ is continuously
differentiable on $\mathbb{R}^{n}$ for $\forall x\in S$. At this time, we have
$\partial_y\phi(x,\cdot)=\{\nabla_y\phi(x,\cdot)\}.$

\begin{proposition}\label{Proposition 3.4}
	In the $EP(\phi,S)$, suppose $\bar{S}\subset S$ is a closed set, and $\{x^{k}\} \subset S$ satisfies
	\begin{eqnarray}\label{eq3.11}
		\lim\limits_{k\rightarrow\infty}\| \nabla_{y}\phi
		(x^k,x^k)-\nabla_{y}\phi (P_{\bar{S}} (x^{k}),P_{\bar{S}} (x^{k}))
		\|=0.
	\end{eqnarray}
	
	If $\bar{S}$ is weak sharp, then $\bar{S}$ is
	augmented weakly sharp with respect to $\{ x^{k} \}$.
\end{proposition}
\begin{proof}Let $K=\{k\mid x^{k} \notin\bar{S}\}$ be an infinite
	sequence, and set
	$$H(z)=\nabla_{y}\phi(z,z),\quad\forall z \in \bar{S}.$$
	By \eqref{eq3.4}, we know that $(a)$ holds in Definition \ref{Definition 3.4}, and by \eqref{eq3.11}, we immediately obtain that
	$$\limsup\limits_{k\in K,~k \rightarrow\infty}\psi_{k}=\limsup\limits_{k\in K,~k \rightarrow\infty}
	\frac{1}{\parallel
		x^{k}-P_{\bar{S}}(x^{k})\|}\langle\nabla_{y}\phi(x^{k},x^{k})
	-H(P_{\bar{S}}(x^{k})),x^{k}-P_{\bar{S}}(x^{k})\rangle=0,$$
	i.e., $(b)$ holds in Definition \ref{Definition 3.4}.
\end{proof}

\begin{proposition}\label{Proposition 3.5}
	In $EP(\phi,S)$, suppose $\bar{S}\subset S$ is a closed set, $\{x^{k}\}\subset S$, $\{\nabla_{y}\phi(x^{k},x^{k})\}$ is bounded and any one of its accumulations $\bar{p}$ satisfies $-\bar{p}\in \text{int} G$. Then $\bar{S}$ is augmented weakly sharp with respect to $\{ x^{k} \}$.
\end{proposition}
\begin{proof} Let $K=\{k\mid x^{k}\notin \bar{S}\}$ be an infinite
	sequence. According to the hypotheses, there must be an accumulation
	$\bar{p}$ of $\{\nabla_{y}\phi(x^{k},x^{k})\}_{k\in K}$ satisfying
	$-\bar{p}\in \text{int} G$, i.e., there exists a constant $\alpha>0$ such
	that
	\begin{eqnarray}\label{eq3.12}
		\alpha B\subset\bar{p}+[T_{S}(z)\cap\hat{N}_{\bar{S}}(z)]^{\circ},\quad\forall
		z\in\bar{S}.
	\end{eqnarray}
	Letting $H(z)=\bar{p}$, $\forall z\in\bar{S}$, by \eqref{eq3.12} we know that $(a)$ holds in Definition \ref{Definition 3.4}. Furthermore let $K_{0}\subseteq K$ such that
	\begin{eqnarray}\label{eq3.13}
		\lim\limits_{k\in
			K_{0},k\rightarrow\infty}\nabla_{y}\phi(x^{k},x^{k})=\bar{p}.
	\end{eqnarray}
	By \eqref{eq3.13}, we immediately get that
	\begin{eqnarray*}
		\limsup\limits_{k\in K,~k \rightarrow\infty}\psi_{k}&\geq&\limsup\limits_{k\in K_{0},k \rightarrow\infty}\psi_{k}\\
		&=&\lim\limits_{k\in K_{0},k\rightarrow\infty}\frac{1}{\parallel
			x^{k}-P_{\bar{s}}(x^{k})\|}\langle\nabla_{y}\phi(x^{k},x^{k})
		-\bar{p},x^{k}-P_{\bar{S}}(x^{k})\rangle\\
		&=&0,
	\end{eqnarray*}
	i.e., $(b)$ holds in Definition \ref{Definition 3.4}.
\end{proof}

Finally, we give the relation of strong non-degeneracy and augmented
weak sharpness.

\begin{proposition}\label{Proposition 3.6}
	In the $EP(\phi,S)$, suppose $\nabla_{y}\phi(x,x)$ is continuous over $S$, $\{x^{k}\}\subset S$ is bounded and any of its accumulations is strongly non-degenerate. Then $\bar{S}$ is augmented weakly sharp with respect to $\{x^{k}\}$.
\end{proposition}
\begin{proof}Let $K=\{k\mid x^{k}\notin\bar{S}\}$ be an infinite
	sequence. Then there must be an accumulation
	point $\bar{x}$ of $\{x^{k}\}_{k\in K}$. Suppose $K_{0}\subseteq K$ such that
	\begin{eqnarray}\label{eq3.14}
		\lim\limits_{k\in K_{0},k\rightarrow\infty}x^{k}=\bar{x}.
	\end{eqnarray}
	By the assumptions of the strong non-degeneracy of
	$\bar{x}$ and the continuity of $\nabla_{y}\phi(x,x)$, and [\cite{wang2008two}, Proposition 5.1], we know that $\bar{x}$ is an isolated point of $\bar{S}$, Thus we have $T_{\bar{S}}(\bar{x})=\{0\},\hat{N}_{\bar{S}}(\bar{x})=T_{\bar{S}}\bar{x})^{\circ}=\mathbb{R}^{n}$. Therefore, we obtain that
	\begin{eqnarray}\label{eq3.15}
		[T_{S}(\bar{x})\cap\hat{N}_{\bar{S}}(\bar{x})]^{\circ}=T_{S}(\bar{x})^{\circ}=N_{S}(\bar{x}).
	\end{eqnarray}
	According to \eqref{eq3.5} and \eqref{eq3.15}, we know that there exists a
	constant $\alpha>0$ such that
	\begin{eqnarray}\label{eq3.16}
		\alpha B\subset\nabla_{y}\phi(\bar{x},\bar{x})+N_{S}(\bar{x})=\nabla_{y}\phi(\bar{x},
		\bar{x})+[T_{S}(\bar{x})\cap\hat{N}_{\bar{S}}(\bar{x})]^{\circ}.
	\end{eqnarray}

	Now we define the augmented mapping
	\begin{eqnarray}\label{eq3.17}
		H(z)=\left\{\begin{array}{ll}
			\nabla_{y}\phi(\bar{x},\bar{x}),\qquad z=\bar{x},\\
			\mathbb{R}^{n},\qquad z\neq\bar{x}.\\
		\end{array} \right.
	\end{eqnarray}
	According to \eqref{eq3.16} and \eqref{eq3.17}, we can immediately obtain that
	$$\alpha B\subset
	H(z)+[T_{S}(z)\cap\hat{N}_{\bar{S}}(z)]^{\circ}, \forall
	z\in\bar{S},$$
	i.e., $(a)$ holds in Definition \ref{Definition 3.4}. Since $x$ is an isolated point of $\bar{S}$, by \eqref{eq3.14}, for all sufficiently large $k\in K_{0}$ it holds that $P_{\bar{S}}(x^{k})=\bar{x}$. Therefore, according to \eqref{eq3.14}, \eqref{eq3.17}, and the continuity of $\nabla_{y}\phi(x,x)$, we obtain that
	\begin{eqnarray*}
		\limsup\limits_{k\in K,~k \rightarrow\infty}\psi_{k}&\geq&\limsup\limits_{k\in K_{0},k \rightarrow\infty}\psi_{k}\\
		&=&\lim\limits_{k\in K_{0},k\rightarrow\infty}\frac{1}{\parallel
			x^{k}-\bar{x}\|}\langle\nabla_{y}\phi(x^{k},x^{k})-\nabla_{y}\phi(\bar{x},
		\bar{x}),x^{k}-\bar{x}\rangle\\
		&=&0
	\end{eqnarray*}
	i.e., $(b)$ holds in Definition \ref{Definition 3.4}.
\end{proof}

\section{Some examples}

In this section, we give some examples of $EP(\phi,S)$ to show that
the solution set $S$ of $EP(\phi,S)$ does not satisfy the weak
sharpness but the augmented weak sharpness.

\begin{example}\label{Example 4.1}
	Consider the following mathematical programming problems $(MP)$ as a special case of $EP(\phi,S)$:
	$$\min\limits_{x\in S} f(x)=\max\{0,x_{1}x_{2}\}+\sum\limits_{i=1}^{2}\max\{0,x_{i}\},$$
	where
	$$S=\{x\in \mathbb{R}^{2}\mid x_{1}\leq1,~x_{2}\leq1\},$$
	$$\bar{S}=\{x\in \mathbb{R}^{2}\mid x_{1}=0,~x_{2}\leq0\}\cup\{x\in \mathbb{R}^{2}\mid x_{1}\leq0,~x_{2}=0\}.$$
	This is a non-smooth and non-convex programming problem, and the
	solution set $\bar{S}$ of it is non-convex. By Example \ref{Example 2.1}, we know that
	$$\phi(x,y)=f(y)-f(x),\quad (x,y)\in \mathbb{R}^{2}\times \mathbb{R}^{2}.$$
	Therefore, $\partial_{y} \phi(x,\cdot)=\partial
	f(\cdot).$ Now, let $(t)^{+}=\max\{0,t\}.$
	
	When $x\in S\setminus \bar{S}$, we have
	\begin{eqnarray}\label{eq4.1}
		\partial f(x_{1},x_{2})=\left\{\begin{array}{ll}
			(x_{2}+(\frac{x_{1}}{|x_{1}|})^{+},x_{1}+(\frac{x_{2}}{|x_{2}|})^{+}),\quad &\text{if}\quad x_{1}x_{2}>0\\
			((\frac{x_{1}}{|x_{1}|})^{+},(\frac{x_{2}}{|x_{2}|})^{+}),\quad &\text{if}\quad x_{1}x_{2}< 0\\
			\{u\in \mathbb{R}^{2}\mid u_{1}\in[\frac{x_{1}}{\|x\|},x_{2}+1],\quad &\text{if}\quad x_{1}x_{2}=0\\
			\qquad\qquad ~~ u_{2}\in[\frac{x_{2}}{\|x\|},x_{1}+1]\},&\qquad\text{and} \quad x_{1}+x_{2}>0.
		\end{array} \right.
	\end{eqnarray}

	When $x\in\bar{S}$, we have
	\begin{eqnarray}\label{eq4.2}
		\partial f(x_{1},x_{2})=\left\{\begin{array}{ll}
			\{u\in \mathbb{R}^{2}\mid u_{1}\in[0,1],u_{2}\in[0,1]\},\quad &\text{if}\quad x=(0,0)\\
			\{u\in \mathbb{R}^{2}\mid u_{1}\in[x_{2},1],u_{2}=0\},\quad &\text{if}\quad x_{1}=0,~x_{2}<0.\\
			\{u\in
			\mathbb{R}^{2}\mid u_{1}=0,u_{2}\in[x_{1},1]\},\quad &\text{if}\quad x_{1}<0,~x_{2}=0.\end{array} \right.
	\end{eqnarray}
	
	Note that $\bar{S}\subset \text{int}S$, for all $x\in\bar{S}$, we have
	$T_{S}(x)=\mathbb{R}^{2}$. So, for $\forall x\in \bar{S}$, we get that
	\begin{eqnarray}\label{eq4.3}
		[T_{S}(x)\cap \hat{N}_{\bar{S}}(x)]^{\circ}&=&\hat{N}_{\bar{S}}(x)^{\circ}\nonumber\\
		&=&\left\{\begin{array}{ll}
			\{\xi\in \mathbb{R}^{2}\mid \xi_{1}\leq0,\xi_{2}\leq0\},\quad &\text{if}\quad x=(0,0)\\
			\{\xi\in \mathbb{R}^{2}\mid \xi_{1}=0,\xi_{2}\in(-\infty,+\infty)\},
			\quad&\text{if}\quad x_{1}=0,~x_{2}<0\\
			\{\xi\in \mathbb{R}^{2}\mid \xi_{1}\in(-\infty,+\infty),\xi_{2}=0\},\quad&\text{if}\quad x_{1}<0,~x_{2}=0
		\end{array} \right.
	\end{eqnarray}
	
	By \eqref{eq4.2}, the second and third formula of \eqref{eq4.3}, we obtain that
	\begin{eqnarray}\label{eq4.4}
		\partial f(x)+\hat{N}_{\bar{S}}(x)^{\circ}=\left\{\begin{array}{ll}
			\{\xi\in \mathbb{R}^{2}\mid \xi_{1}\in[x_{2},1],\xi_{2}\in(-\infty,+\infty)\},\quad &
			\text{if}\quad x_{1}=0,~x_{2}<0\\
			\{\xi\in \mathbb{R}^{2}\mid \xi_{1}\in(-\infty,+\infty),\xi_{2}\in[x_{1},1]\},
			\quad &\text{if}\quad x_{1}<0,~x_{2}=0.\end{array} \right.
	\end{eqnarray}
	
	By \eqref{eq4.4} we know that the model $\alpha >0$ is not a constant in Definition \ref{Definition 3.2}, which is related to $x\in\bar{S}$, and when
	$\|x\| \rightarrow 0$, $\alpha\rightarrow 0^{+}$, i.e., the constant
	$\alpha >0$ does not exist. Therefore, $\bar{S}$ is not a weak sharp set.
	
	Now, take small enough $\varepsilon>0$, and let
	$$\Omega_{\varepsilon}=\{x\in \mathbb{R}^{2}|-\varepsilon
	<x_{1}<0,~-\varepsilon<x_{2}<0\}.$$
	
	Below we will prove that $\bar{S}$ is an augmented weak sharp set
	with respect to arbitrary sequences $\{x^{k}\}\subset
	S\setminus\Omega_{\varepsilon}$. Let $K=\{k\mid x^{k}\notin\bar{S}\}$ be
	an infinite sequence. We introduce the augmented mapping
	$H:\bar{S}\Rightarrow \mathbb{R}^{2}$ as follows:
	\begin{eqnarray}\label{eq4.5}
		H(x_{1},x_{2})=\left\{\begin{array}{ll}
			(\varepsilon,\varepsilon),\quad &\text{if}\quad x=(0,0)\\
			\{v\in \mathbb{R}^{2}||v_{1}\mid \leq\varepsilon,v_{2}=0\},\quad &\text{if}\quad x_{1}=0,~x_{2}<0.\\
			\{v\in \mathbb{R}^{2}|v_{1}=0,|v_{2}\mid \leq\varepsilon\},\quad &\text{if}\quad x_{1}<0,~x_{2}=0.\end{array} \right.
	\end{eqnarray}
	
	According to \eqref{eq4.3}, \eqref{eq4.5}, and the conditions $(a)$ in Definition \ref{Definition 3.4} holds, i.e., there exists a constant $\alpha
	>0$ such that for all $x\in\bar{S}$,
	$$\alpha B\subset H(x)+\hat{N}_{\bar{S}}(x)^{\circ}.$$
	
	For $k\in K$, let
	$$\psi_{k}=\frac{1}{\|x^{k}-P_{\bar{S}}(x^{k})\|}\langle
	u^{k}-v^{k}, x^{k}-P_{\bar{S}}(x^{k})\rangle,$$
	where $u^{k}\in
	\partial f(x^{k}),v^{k}\in H(P_{\bar{S}}(x^{k}))$, and
	\begin{eqnarray}\label{eq4.6}
		P_{\bar{S}}(x^{k})=\left\{\begin{array}{ll}
			(0,0),\quad &\text{if}\quad x_{1}^{k}\geq0,~x_{2}^{k}\geq 0,\quad\text{and}\quad x_{1}^{k}+x_{2}^{k}>0\\
			(0,x_{2}^{k}),\quad &\text{if}\quad x_{1}^{k}>0,~x_{2}^{k}<0,\quad\text{or}\quad x_{2}^{k}\leq x_{1}^{k}<0\\
			(x_{1}^{k},0),\quad &\text{if}\quad x_{1}^{k}<0,~x_{2}^{k}>0,\quad\text{or}\quad x_{1}^{k}\leq x_{2}^{k}<0.\end{array} \right. 
	\end{eqnarray}
	
	According to \eqref{eq4.1}, \eqref{eq4.5}, and \eqref{eq4.6}, we can easily prove that
	$\{x^{k}\}_{k\in K}\subset S\backslash \Omega_{\varepsilon}$
	satisfies the condition $(b)$ in Definition \ref{Definition 3.4}. For simplicity,
	we only prove the case that $\{x^{k}\}_{k\in K}$ lies in the third
	quadrant. At the same time, considering $x^{k}\notin
	\Omega_{\varepsilon}$, we have
	
	$(i)$ when $x_{2}^{k}\leq x_{1}^{k}<0$ and $x_{2}^{k}\leq
	-\varepsilon$, by the first item of \eqref{eq4.1}, \eqref{eq4.6}, and the second
	item of \eqref{eq4.5}, we have
	$$\psi_{k}=\frac{1}{|x_{1}^{k}|}(x_{2}^{k}-v_{1}^{k})x_{1}^{k}=|x_{2}^{k}|+
	v_{1}^{k}\geq\varepsilon-|v_{1}^{k}\mid \geq0,$$
	
	$(ii)$ when $x_{1}^{k}\leq x_{2}^{k}<0$ and $x_{1}^{k}\leq
	-\varepsilon$, by the first item of \eqref{eq4.1}, \eqref{eq4.6}, and the third
	item of \eqref{eq4.5}, we have
	$$\psi_{k}=\frac{1}{|x_{2}^{k}|}(x_{1}^{k}-v_{2}^{k})x_{2}^{k}=|x_{1}^{k}|+v_{2}^{k}
	\geq\varepsilon-|v_{2}^{k}\mid \geq0.$$

	By $(i),(ii)$, we get that
	$$\limsup\limits_{k\in K,~k
		\rightarrow\infty}\psi_{k}\geq0,$$
	i.e., $(b)$ in Definition \ref{Definition 3.4} holds.
	
	In addition, we note that Remark \ref{Remark 3.1} is verified through Example \ref{Example 4.1}. As in this example, the regular normal cone of $\bar{S}$ at $(0,0)$ is
	$$\hat{N}_{\bar{S}}(0,0)=\{\xi\in
	\mathbb{R}^{2}\mid \xi_{1}\geq0,\xi_{2}\geq0\},$$
	while the normal cone under the general meaning is
	$$N_{\bar{S}}(0,0)
	=\hat{N}_{\bar{S}}(0,0)\cup\{\xi\in \mathbb{R}^{2}\mid \xi_{1}
	\in(-\infty,0),\xi_{2}=0\} \cup\{\xi\in
	\mathbb{R}^{2}\mid \xi_{1}=0,\xi_{2}\in(-\infty,0)\}.$$
	Here, $N_{\bar{S}}(0,0)$ is a closed cone, but not a convex cone. Furthermore, according to $T_{S}(0,0)=\mathbb{R}^{2}$, it follows that
	$$[T_{S}(0,0)\cap\hat{N}_{\bar{S}}(0,0)]^{\circ}=\hat{N}_{\bar{S}}
	(0,0)^{\circ}=\{\xi\in k^{2}\mid \xi_{1}\leq0,\xi_{2}\leq0\},$$
	$$[T_{S}(0,0)\cap N_{\bar{S}}(0,0)]^{\circ}=N_{\bar{S}}(0,0)^{\circ}=\{(0,0)\}.$$
	The advantage of the regular cone has been shown here compared with
	the normal cone under the general meaning.
\end{example}
\begin{example}\label{Example 4.2}
	Consider the following variational inequality problem $VIP(F,S)$ as a special case of $EP(\phi,S)$:
	$$F(x)=(\cos x_{1},e^{x_{2}}),$$
	where $S=\{x\in \mathbb{R}^{2}|0\leq x_{1}\leq
	\frac{\pi}{2},x_{2}\geq0\}$, $\bar{S}=\{(0,0),(\frac{\pi}{2},0)\}$. By Example \ref{Example 2.2}, we have $\phi(x,y)=\langle F(x),y-x \rangle,
	(x,y)\in \mathbb{R}^{2}\times \mathbb{R}^{2}$. Then, we have $\partial_{y}\phi(x,\cdot)=F(x)$.
	
	When $x\in\bar{S}$, we have
	\begin{eqnarray}\label{eq4.7}
		F(x)=\left\{\begin{array}{ll}
			(1,1),\qquad x=(0,0)\\
			(0,1),\qquad x=(\frac{\pi}{2},0),\end{array} \right.
	\end{eqnarray}
	and
	\begin{eqnarray}\label{eq4.8}
		[T_{S}(x)\cap N_{\bar{S}}(x)]^{\circ}=N_{S}(x)=\left\{\begin{array}{ll}
			\{\xi\in \mathbb{R}^{2}\mid \xi_{1}\leq0,\xi_{2}\leq 0\},
			\quad x=(0,0)\\
			\{\xi\in \mathbb{R}^{2}\mid \xi_{1}\geq0,\xi_{2}\leq 0\},\quad x=(\frac{\pi}{2},0).\end{array} \right.
	\end{eqnarray}
	
	This is a non-monotonic variational inequality problem. By
	\eqref{eq4.7} and \eqref{eq4.8}, one can see that $\bar{S}$ is not a weak sharp set.
	
	Next, we will prove that $\bar{S}$ is an augmented weak sharp set
	with respect to the sequence $\{x^{k}\}\subset S$ which satisfies
	the following conditions:
	
	$(i)~\{x^{k}\}\subset\{x\in
	\mathbb{R}^{2}\mid x_{1}\leq\frac{\pi}{4}\}\cup\{x\in
	\mathbb{R}^{2}\mid x_{1}\geq\frac{\pi}{4},x_{1}+x_{2}\geq\frac{\pi}{2}\};$
	
	$(ii)~\lim\limits_{k\rightarrow\infty}\text{dist}(x^{k},\bar{S})=0.$
	
	For this purpose, let $K=\{k\mid x^{k}\notin \bar{S}\}$ be an infinite
	sequence. Take $\lambda\in(0,\frac{1}{2})$, we introduce the
	augmented mapping $H:\bar{S}\rightarrow \mathbb{R}^{2}$ as follows:
	\begin{eqnarray}\label{eq4.9}
		H(x)=\left\{\begin{array}{ll}
			(\lambda,\lambda),\quad&\text{if}\quad x=(0,0)\\
			(-\lambda,\lambda),\quad&\text{if}\quad x=(\frac{\pi}{2},0).
		\end{array} \right.
	\end{eqnarray}
	
	By \eqref{eq4.8} and \eqref{eq4.9}, we obtain that the condition $(a)$ in Definition \ref{Definition 3.4} holds. Now, for $k\in K$, let
	$$\psi_{k}=\frac{1}{\|x^{k}-P_{\bar{S}}(x^{k})\|}\langle F(x^{k})-
	H(P_{\bar{S}}(x^{k})),x^{k}-P_{\bar{S}}(x^{k})\rangle,$$
	where
	\begin{eqnarray}\label{eq4.10}
		P_{\bar{S}}(x^{k})=\left\{\begin{array}{ll}
			(0,0),\quad&\text{if}\quad x_{1}^{k}\leq\frac{\pi}{4}\\
			(\frac{\pi}{2},0),\quad&\text{if}\quad x_{1}^{k}\geq\frac{\pi}{4}.\end{array} \right.
	\end{eqnarray}
	
	By the condition $(ii)$, the accumulations of the bounded sequence
	$\{x^{k}\}_{k\in K}$ are only possibly $\bar{x}=(0,0)$ or
	$\bar{x}=(\frac{\pi}{2},0)$. Without loss of generality, let
	$\bar{x}=(\frac{\pi}{2},0)$ be one of its accumulations. Then there
	exists an infinite subsequence $K_{0}\subseteq K$ such that
	\begin{eqnarray}\label{eq4.11}
		\lim\limits_{k\in K_{0},k\rightarrow\infty}x^{k}=(\frac{\pi}{2},0).
	\end{eqnarray}
	
	According to \eqref{eq4.9}, \eqref{eq4.10}, and \eqref{eq4.11}, when $k\in K_{0}$ is big
	enough, we obtain that
	\begin{eqnarray*}
		\psi_{k}&=&\frac{1}{\|x^{k}-\bar{x}\|}[(\cos x_{1}^{k}+\lambda)(x_{1}^{k}-\frac{\pi}{2})+(e^{x_{2}^{k}}-\lambda)x_{2}^{k}]\\
		&\geq&\frac{1}{\|x^{k}-\bar{x}\|}[\cos x_{1}^{k}(x_{1}^{k}-\frac{\pi}{2})+(e^{x_{2}^{k}}-2\lambda)x_{2}^{k}]\qquad [by\quad (i)]\\
		&\geq&\frac{1}{\|x^{k}-\bar{x}\|}(x_{1}^{k}-\frac{\pi}{2})\cos
		x_{1}^{k}\qquad [by~
		\lambda\in(0,\frac{1}{2})]
	\end{eqnarray*}
	Furthermore, by \eqref{eq4.11}, we immediately get that $$\limsup\limits_{k\in
		K,~k\rightarrow\infty}\psi_{k}\geq\lim\limits_{k\in
		K_{0},k\rightarrow\infty}\frac{1}{\|x^{k}-\bar{x}\|}(x_{1}^{k}-\frac{\pi}{2})\cos
	x_{1}^{k}=0,$$ i.e., $(b)$ in Definition \ref{Definition 3.4} holds.
\end{example}
\begin{example}\label{Example 4.3}
	Consider $EP(\phi,S)$, where
	$$\phi(x,y)=(e^{y_{1}^{2}}+x_{2}^{2})-(e^{x_{1}^{2}}+y_{2}^{2}),$$
	$S=S_{1}\times S_{2}=[0,1]\times[-1,1]$, and $\bar{S}=\{(0,1),(0,-1)\}.$
	
	By Example \ref{Example 2.3}, one can see that this is a saddle point problem $(SPP)$, i.e., find $\bar{x}=(\bar{x_{1}},\bar{x_{2}})\in
	S_{1}\times S_{2}$ such that
	$$\varphi(\bar{x}_{1},y_{2})\leq\varphi(\bar{x}_{1},\bar{x}_{2})\leq\varphi(y_{1},\bar{x}_{2}),\quad \forall(y_{1},y_{2})\in S_{1}\times S_{2},$$
	where $\varphi(x_{1},x_{2})=e^{x_{1}^{2}}+x_{2}^{2}$. It can be easily seen that
	\begin{eqnarray}\label{eq4.12}
		\nabla_{y}\phi(x,x)=(2x_{1}e^{x_{1}^{2}},-2x_{2}),
	\end{eqnarray}
	and
	\begin{eqnarray}\label{eq4.13}
		\nabla_{y}\phi(x,x)=\left\{\begin{array}{ll}
			(0,-2),\quad &\text{if}\quad x=(0,1)\\
			(0,2),\quad&\text{if}\quad x=(0,-1).
		\end{array} \right.
	\end{eqnarray}
	
	When $x\in \bar{S}$, we have
	\begin{eqnarray}\label{eq4.14}
		[T_{S}(x)\cap N_{\bar{S}}(x)]^{\circ}=N_{S}(x)=\left\{\begin{array}{ll}
			\{\xi\in \mathbb{R}^{2}\mid \xi_{1}\leq0,\xi_{2}\geq 0\},\quad &\text{if}\quad x=(0,1)\\
			\{\xi\in \mathbb{R}^{2}\mid \xi_{1}\leq0,\xi_{2}\leq 0\}.\quad &\text{if}\quad x=(0,-1).
		\end{array} \right.
	\end{eqnarray}
	
	By \eqref{eq4.13} and \eqref{eq4.14}, one can see that $\bar{S}$ is not a weak sharp set and that $\bar{S}$ is not a strongly non-degenerate set.
	
	Next, we will prove that $\bar{S}$ is an augmented weak sharp set
	with respect to the sequence $\{x^{k}\}\subset S$ which satisfies
	the following conditions:
	
	$(i)~\{x^{k}\}\subset\{x\in \mathbb{R}^{2}\mid x_{1}+|x_{2}| \leq 1\};$
	
	$(ii)~\lim\limits_{k\rightarrow\infty}\text{dist}(x^{k},\bar{S})=0.$
	
	For this purpose, let $K=\{k\mid x^{k}\notin \bar{S}\}$ be an infinite
	sequence. Take $\lambda\in(0,1)$, we introduce the augmented mapping
	$H:\bar{S}\rightarrow \mathbb{R}^{2}$ as follows:
	\begin{eqnarray}\label{eq4.15}
		H(x)=\left\{\begin{array}{ll}
			(\lambda,-\lambda),\quad&\text{if}\quad x=(0,1)\\
			(\lambda,\lambda),\quad&\text{if}\quad x=(0,-1).
		\end{array} \right.
	\end{eqnarray}
	
	By \eqref{eq4.14} and \eqref{eq4.15}, we obtain that the condition $(a)$ in Definition \ref{Definition 3.4} holds. Now we will prove the condition $(b)$ in Definition \ref{Definition 3.4} also holds. By $(ii)$ one can see that the accumulations of $\{x^{k}\}$ are $\bar{x}=(0,1)$ or $\bar{x}=(0,-1)$. Without loss of generality, assume that there exists a sequence $k_{0}\subseteq K$ such that
	\begin{eqnarray}\label{eq4.16}
		\lim\limits_{K\in K_{0},k\rightarrow\infty}x^{k}=(0,1).
	\end{eqnarray}
	
	For $k\in K$, let
	$$\psi_{k}=\frac{1}{\parallel x^{k}-\bar{x}\|}\langle\nabla_{y}\phi(x^{k},x^{k})-H(P_{\bar{S}}(x^{k})),x^{k}-P_{\bar{S}}(x^{k})\rangle,$$
	where
	\begin{eqnarray}\label{eq4.17}
		P_{\bar{S}}(x^{k})=\left\{\begin{array}{ll}
			(0,1),\quad &\text{if}\quad x^{k}\geq 0\\
			(0,-1),\quad &\text{if}\quad x_{2}^{k}\leq 0.
		\end{array} \right.
	\end{eqnarray}
	
	According to \eqref{eq4.12}, \eqref{eq4.15}, \eqref{eq4.16}, and \eqref{eq4.17}, when $k\in K_{0}$
	is big enough, we obtain that
	\begin{eqnarray*}
		\psi_{k}&=&\frac{1}{\|x^{k}-\bar{x}\|}[(2x_{1}^{k}e^{(x_{1}^{k})^{2}}-\lambda)x_{1}^{k}+(\lambda-2x_{2}^{k})(x_{2}^{k}-1)]\\
		&\geq&\frac{1}{\|x^{k}-\bar{x}\|}[(2x_{2}^{k}-\lambda)(1-x_{2}^{k})-\lambda x_{1}^{k}]\\
		&\geq&\frac{1}{\|x^{k}-\bar{x}\|}[(2x_{2}^{k}-\lambda)(1-x_{2}^{k})-\lambda (1-x_{2}^{k})]\quad (by~ (i))\\
		&=&\frac{1}{\|x^{k}-\bar{x}\|}2(x_{2}^{k}-\lambda)(1-x_{2}^{k})\geq
		0\qquad (by~ \lambda\in(0,1))
	\end{eqnarray*}
	Thus we have
	$$\limsup\limits_{k\in K,~k\rightarrow\infty}\psi_{k}\geq \lim\limits_{k\in K_{0},k\rightarrow\infty}
	\frac{1}{\|x^{k}-\bar{x}\|}2(x_{2}^{k}-\lambda)(1-x_{2}^{k})\geq
	0,$$
	i.e., $(b)$ in Definition \ref{Definition 3.4} holds.
\end{example}
\begin{example}\label{Example 4.4}
	Considering $EP(\phi,S)$, where
	$$\phi(x,y)=(e^{y_{1}-x_{2}}-e^{x_{2}-y_{1}}-e^{x_{1}-x_{2}}+e^{x_{2}-x_{1}})+(e^{y_{2}^{2}-x_{1}^{2}}-e^{x_{2}^{2}-x_{1}^{2}}),$$
	$S=S_{1}\times S_{2}=[0,1]\times[0,1]$, $\bar{S}=\{(0,0)\},\quad \bar{x}=(0,0).$
	By Example \ref{Example 2.4}, one can see that this is a $Nash$ equilibrium
	problem $(NEP)$, i.e., find $\bar{x}\in S$ such that
	$$f_{1}(\bar{x})\leq f_{1}(y_{1},\bar{x}_{2}),\quad \forall y_{1}\in
	S_{1},$$
	$$f_{2}(\bar{x})\leq f_{2}(\bar{x}_{1},y_{2}),\quad \forall y_{2}\in S_{2},$$
	where
	$f_{1}(x)=e^{x_{1}-x_{2}}-e^{x_{2}-x_{1}},~f_{2}(x)=e^{x_{2}^{2}-x_{1}^{2}}-1.$
	It is easy to see that
	\begin{eqnarray}
		\nabla_{y}\phi(x,x)&=&(e^{x_{1}-x_{2}}+e^{x_{2}-x_{1}},2x_{2}e^{x_{2}^{2}-x_{1}^{2}}),\label{eq4.18}\\
		\nabla_{y}\phi(0,0)&=&(2,0),\label{eq4.19}\\
		{[T_{S}(0,0)\cap N_{\bar{S}}(0,0)]}^{\circ}&=&N_{S}(0,0)
		=\{\xi\in \mathbb{R}^{2}\mid \xi_{1}\leq0,~\xi_{2}\leq 0\}.\label{eq4.20}
	\end{eqnarray}
	
	By \eqref{eq4.19} and \eqref{eq4.20}, we obtain that $\bar{S}$ is not a weak sharp
	set, i.e., the point $(0,0)$ is not a strongly non-degenerate point.
	
	Now we will prove that $\bar{S}$ is an augmented weak sharp set for arbitrary sequences $\{x^{k}\}\subset S\cap\{x\in
	\mathbb{R}^{2}\mid x_{2}\leq x_{1}\}$. For this purpose, let
	$K=\{k\mid x^{k}\notin\bar{S}\}$ be an infinite sequence. Take
	$\lambda\in(0,\frac{1}{2})$, we introduce the augmented mapping
	$H:\bar{S}\rightarrow \mathbb{R}^{2}$ as follows:
	\begin{eqnarray}\label{eq4.21}
		H(0,0)=(\lambda,\lambda).
	\end{eqnarray}
	
	By \eqref{eq4.20} and \eqref{eq4.21}, one can see that the condition $(a)$ in Definition \ref{Definition 3.4} holds. Furthermore, according to \eqref{eq4.18} and \eqref{eq4.21}, for $k\in K$, we obtain that
	\begin{eqnarray*}
		\psi_{k}&:=&\frac{1}{\|x^{k}\|}\langle\nabla_{y}\phi(x^{k},x^{k})-H(0,0),x^{k}\rangle\\
		&=&\frac{1}{\|x^{k}\|}[(e^{x_{1}^{k}-x_{2}^{k}}+e^{x_{2}^{k}-x_{1}^{k}}-\lambda)x_{1}^{k}
		+(2x_{1}^{k}e^{(x_{2}^{k})^{2}-(x_{1}^{k})^{2}}-\lambda)x_{2}^{k}]\\
		&\geq&\frac{1}{\|x^{k}\|}[(e^{x_{1}^{k}-x_{2}^{k}}+e^{x_{2}^{k}-x_{1}^{k}}-\lambda)x_{1}^{k}-\lambda x_{2}^{k}]\\
		&\geq&\frac{1}{\|x^{k}\|}(e^{x_{1}^{k}-x_{2}^{k}}-2\lambda)x_{1}^{k}\quad [by\quad x_2^k\leq x_1^k]\\
		&\geq &\frac{1}{\|x^{k}\|}(1-2\lambda)x_{1}^{k}\geq0.\qquad [by\quad \lambda\in (1,\frac{1}{2})]\\
	\end{eqnarray*}
	Thus the condition $(b)$ in Definition \ref{Definition 3.4} also holds.
\end{example}

The following example has certain characteristics.

\begin{example}\label{Example 4.5}
	Considering the following non-convex programming problem $(MP)$:
	$$\min\limits_{x\in S}~f(x)=\sin x_{1}+\max\{0,x_{2}\},$$
	where $S=\{x\in \mathbb{R}^{2}|0\leq x_{1}\leq\frac{\pi}{3},-1\leq x_{2}\}$, $\bar{S}=\{x\in \mathbb{R}^{2}\mid x_{1}=0,-1\leq x_{2}\leq 0\}$, 
	and
	\begin{eqnarray}\label{eq4.22}
		\partial f(x_{1},x_{2})=\left\{\begin{array}{ll}
			(\cos x_{1},1),\quad&\text{if}\quad x_{2}>0\\
			\{u\in \mathbb{R}^{2}\mid u_{1}=\cos x_{1},u_{2}\in[0,1]\},\quad&\text{if}\quad x_{2}=0\\
			(\cos x_{1},0),\quad&\text{if}\quad x_{2}<0.
		\end{array} \right.
	\end{eqnarray}
	Note that $\bar{S}$ is convex, we have
	\begin{eqnarray}\label{eq4.23}
		[T_{S}(x)\cap N_{\bar{S}}(x)]^{\circ}=\left\{\begin{array}{ll}
			\{\xi\in \mathbb{R}^{2}\mid \xi_{1}\leq0,\xi_{2}\leq 0\},\quad&\text{if}\quad x=(0,0)\\
			\{\xi\in \mathbb{R}^{2}\mid \xi_{1}\leq0,~\xi_{2}\in(-\infty,\infty)\}, \quad &\text{if}\quad x_{1}=0,~-1\leq x_{2}<0.
		\end{array} \right.
	\end{eqnarray}
	When $x\in\bar{S}$, from \eqref{eq4.22} one can see
	\begin{eqnarray}\label{eq4.24}
		\partial f(x_{1},x_{2})=\left\{\begin{array}{ll}
			\{u\in \mathbb{R}^{2}\mid u_{1}=1,u_{2}\in[0,1]\},\quad &\text{if}\quad x=(0,0)\\
			(1,0),\quad &\text{if}\quad x_{1}=0,~-1\leq x_{2}<0.
		\end{array} \right.
	\end{eqnarray}
	According to \eqref{eq4.23} and \eqref{eq4.24}, we obtain that $\bar{S}$ is a weak sharp set. Furthermore, if we assume $\{x^{k}\}\subset S$, then
	\begin{eqnarray}\label{eq4.25}
		P_{\bar{S}}(x^{k})=\left\{\begin{array}{ll}
			(0,0),\quad &\text{if}\quad x_{2}^{k}>0\\
			(0,x_{2}^{k}),\quad &\text{if}\quad x_{2}^{k}\leq0.
		\end{array} \right.
	\end{eqnarray}
	
	Define the mapping $H(x)=\partial f(x)$ over $\bar{S}$, and then it
	is easy to prove that $\bar{S}$ is an augmented weak sharp set with
	respect to the sequence $\{x^{k}\}$ which satisfies the condition
	$\lim\limits_{k\rightarrow\infty}dis t(x^{k},\bar{S})=0$.
	
	It is worthwhile to note that by modifying the mapping $\partial
	f(\cdot)$ we can prove $\bar{S}$ is augmented weak sharp with
	respect to arbitrary $\{x^{k}\}\subset S$. For this purpose, let
	$K=\{k\mid x^{k}\notin \bar{S}\}$ be an infinite sequence. Taking
	$\lambda\in(0,\cos\frac{\pi}{3}]$, we introduce the augmented
	mapping $H:\bar{S}\rightarrow \mathbb{R}^{2}$ as follows:
	\begin{eqnarray}\label{eq4.26}
		H(x)=\left\{\begin{array}{ll}
			(\lambda,\lambda),\quad &\text{if}\quad x=(0,0)\\
			(\lambda,0),\quad &\text{if}\quad x_{1}=0,~-1\leq x_{2}<0.\end{array} \right.
	\end{eqnarray}
	
	According to \eqref{eq4.23} and \eqref{eq4.26}, one can see that the condition $(a)$ in Definition \ref{Definition 3.4} holds. Furthermore, for $k\in K$, let
	$$\psi_{k}=\frac{1}{\|x^{k}-P_{\bar{S}}(x^{k})\|}\langle u^{k}-v^{k},x^{k}-P_{\bar{S}}(x^{k})\rangle,$$
	where $u^{k}\in \partial f(x^{k}),v^{k}\in H(P_{\bar{S}}(x^{k}))$.
	By \eqref{eq4.22}, \eqref{eq4.25}, and \eqref{eq4.26}, We immediately get that
	\begin{eqnarray*}
		\psi_{k}=\left\{\begin{array}{ll}
			\frac{1}{\|x^{k}\|}[(\cos x_{1}^{k}-\lambda)x_{1}^{k}+(1-\lambda)x_{2}^{k}]\geq0,\quad&\text{if}\quad x_{1}^{k}\geq0,~x_{2}^{k}>0\\
			\frac{1}{|x_1^{k}|}(\cos x_{1}^{k}-\lambda)x_{1}^{k}\geq0,\quad&\text{if}\quad x_{1}^{k}>0,~x_{2}^{k}\leq0 \end{array} \right.(\text{by~ the~ difinition~ of}~ \lambda).
	\end{eqnarray*}
	Thus the condition $(b)$ in Definition \ref{Definition 3.4} also holds.
\end{example}

The following are some examples similar to the above examples.

1.
\begin{eqnarray*}
	EP(\phi,S) \quad
	\begin{array}{ll}
		\phi(x,y)=x_{1}^{2}y_{1}^{2}+x_{2}^{2}y_{2}^{2}-x_{1}^{4}-y_{2}^{4},\\
		S=\{x\in \mathbb{R}^{2}\mid 0\leq x_{1}\leq
		1,~-1\leq x_{2}\leq 1\},\\
		\bar{S}=\{(0,1),~(0,-1)\}.
	\end{array}
\end{eqnarray*}

2.
\begin{eqnarray*}
	EP(\phi,S) \quad
	\begin{array}{ll}
		\phi(x,y)=e^{y_{1}^{2}-x_{1}^{2}}+y_{2}-x_{2}-1,\\
		S=\{x\in \mathbb{R}^{2}\mid x_{1}\geq 0,~x_{2}\geq 0\},\\
		\bar{S}=\{(0,~0)\}.
	\end{array}
\end{eqnarray*}

3.
\begin{eqnarray*}
	(MP) \quad
	\begin{array}{ll}
		\min\limits_{x\in S}~\left\{f(x)=x_{1}\log(1+x_{2})+\frac{1}{2}(x_{1}-1)^{2}(x_{2}-1)^{2}\right\},\\
		S=\{x\in \mathbb{R}^{2}\mid x_{1}\geq 0,~x_{2}\geq 0,~x_{1}^{2}+x_{2}^{2}\leq 1\},\\
		\bar{S}=\{(1,0),~(0,1)\}.
	\end{array}
\end{eqnarray*}

4.
\begin{eqnarray*}
	(MP) \quad
	\begin{array}{ll}
		\min\limits_{x\in S}~\left\{f(x)=\sin x_{1}\cos x_{2}\right\},\\
		S=\{x\in \mathbb{R}^{2}\mid 0\leq x_{1}\leq
		\frac{\pi}{2},~0\leq x_{2}\leq \frac{\pi}{2}\},\\
		\bar{S}=\{x\in \mathbb{R}^{2}\mid x_{1}=0,~0\leq x_{2}\leq
		\frac{\pi}{2}\}\cup\{x\in \mathbb{R}^{2}\mid 0\leq
		x_{1}\leq \frac{\pi}{2},~x_{2}=\frac{\pi}{2}\}.
	\end{array}
\end{eqnarray*}

5.
\begin{eqnarray*}
	VIP(F,S) \quad
	\begin{array}{ll}
		F(x)=(-x_{1}e^{1-x_{1}^{2}},~x_{2}(x_{2}^{2}-1)),\\
		S=\{x\in \mathbb{R}^{2}\mid -1\leq x_{1}\leq 1,~-1\leq x_{2}\leq 1\},\\
		\bar{S}=\{((-1)^{i},(-1)^{j})\mid i=1,2,~j=1,2\}.
	\end{array}
\end{eqnarray*}

\section{Finite termination of feasible solution sequence}

In this section, under the condition that the solution set $\bar{S}$
of $EP(\phi,S)$ is augmented weak sharp with respect to
$\{x^{k}\}\subset S$ , we present the condition of finite
identification of $\{x^{k}\}$ (see Theorem \ref{Theorem 5.1}). Applying the
result to four special cases of $EP(\phi,S)$ (see Examples
\ref{Example 2.1}-\ref{Example 2.4}), we derive a series of results for finite identification of feasible solution sequences in these cases. In the first two special cases, i.e., the mathematical programming and variational
inequalities problems, these results are generalizations of the
corresponding results in the literature under the condition that
$\bar{S}$ is weakly sharp or strongly non-degenerate. In the last
two special cases, i.e., the saddle point and $Nash$ equilibrium
problems, the finite identification problems of feasible solution
sequences have not been studied by other authors in the literature.

\begin{theorem}\label{Theorem 5.1}\quad In the $EP(\phi,S)$, let $\bar{S}\subset S$ be a closed set, and for $\forall x\in
	S,\partial_{y}\phi(x,x)\neq\emptyset$, $\bar{S}$ is augmented weak
	sharp with respect to $\{x^{k}\}\subset S$. So the following
	conclusions are established.
	
	(i) Suppose \eqref{eq1.2} holds. If $\{x^{k}\}$ terminates finitely to
	$\bar{S}$, then we have
	\begin{eqnarray}\label{eq5.1}
		0\in\liminf_{k\rightarrow\infty}P_{T_{S}(x^{k})}(-\partial_{y}\phi(x^{k},x^{k})).
	\end{eqnarray}

	(ii) If \eqref{eq5.1} holds, then $\{x^{k}\}$ terminates finitely to
	$\bar{S}$.
\end{theorem}
\begin{proof} (i) If $x^{k}\in\bar{S}$, then by \eqref{eq1.2}, we know that
	there exist $u^k\in \partial_y\phi(x^k,x^k)$ such that $-u^k\in
	N_S(x^k)$. Therefore by the convexity of $S$ and projective
	decomposition, we have that $P_{T_{S}(x^{k})}(-u^{k})=0$, i.e.,
	\eqref{eq5.1} holds.
	
	(ii) Suppose \eqref{eq5.1} holds. Now we prove that $\{x^{k}\}$ terminates
	finitely to $\bar{S}$. If not, then $K=\{k\mid x^{k}\notin\bar{S}\}$ is
	an infinite sequence. According to that $\bar{S}$ is augmented
	weakly sharp with respect to $\{x^{k}\}\subset S$, there exist a
	mapping $H:\bar{S}\rightarrow 2^{\mathbb{R}^{2}}$ and a constant $\alpha>0$
	such that
	\begin{eqnarray}\label{eq5.2}
		\alpha B\subset H(z)+[T_{S}(z)\cap\hat{N}_{\bar{S}}(z)]^{\circ}, \forall z\in\bar{S},
	\end{eqnarray}
	and for $\forall u^{k}\in\partial_{y}\phi(x^{k},x^{k})$, $v^{k}\in
	H(P_{\bar S}(x^{k}))$, we have
	\begin{eqnarray}\label{eq5.3}
		\limsup\limits_{k\in K,~k\rightarrow\infty}\frac{1}{\|x^{k}-P_{\bar{S}}(x^{k})\|}\langle u^{k}-v^{k},
		x^{k}-P_{\bar{S}}(x^{k})\rangle\geq 0.
	\end{eqnarray}
	Let
	$z^{k}=P_{\bar{S}}(x^{k})$. Then for $\forall z\in\bar{S}$ we have
	$\|z^{k}-x^{k}\|^{2}\leq \|z-x^{k}\|^{2}$, and
	$$\langle x^{k}-z^{k},z-z^{k}\rangle\leq\frac{1}{2}\|z-z^{k}\|^{2}\\
	=\circ(\|z-z^{k}\|).$$
	Therefore, by the definition of
	$\hat{N}_{\bar{S}}(\cdot)$ we obtain that
	\begin{eqnarray}\label{eq5.4}
		x^{k}-z^{k}\in\hat{N}_{\bar{S}}(z^{k}).
	\end{eqnarray}
	
	Furthermore, by the convexity of $S$, we have
	\begin{eqnarray}\label{eq5.5}
		x^{k}-z^{k}\in T_{S}(z^{k}),~z^{k}-x^{k}\in T_{S}(x^{k}).
	\end{eqnarray}
	According to \eqref{eq5.4} and \eqref{eq5.5}, we immediately get that
	\begin{eqnarray}\label{eq5.6}
		x^{k}-z^{k}\in T_{S}(z^{k})\cap\hat{N}_{\bar{S}}(z^{k}).
	\end{eqnarray}
	Let
	$g_{k}=\frac{x^{k}-z^{k}}{\|x^{k}-z^{k}\|}(k\in K).$ By \eqref{eq5.2} one
	can see that there exists $\bar{v}^{k}\in H(z^{k}),
	~\bar{\xi}^k\in[T_{S}(z^{k})\cap\hat{N}_{\bar{S}}(z^{k})]^{\circ}$
	such that for $\forall k\in K$. We have
	\begin{eqnarray}\label{eq5.7}
		\alpha g_{k}=\bar{v}^{k}+\bar{\xi}^{k}.
	\end{eqnarray}
	By \eqref{eq5.6} and \eqref{eq5.7}, for $\forall k\in K$ we obtain that
	\begin{eqnarray}\label{eq5.8}
		\alpha=\langle\bar{v}^{k},g_{k}\rangle+\langle\bar{\xi}^{k},g_{k}\rangle
		\leq\langle\bar{v}^{k},g_{k}\rangle.
	\end{eqnarray}
	On the other hand, from \eqref{eq5.1} one can see
	$$0\in\liminf_{k\in K,~k\rightarrow\infty}P_{T_{S}(x^{k})}(-\partial_{y}\phi(x^{k},x^{k})).$$
	Therefore, there exists
	$\bar{u}^{k}\in\partial_{y}\phi(x^{k},x^{k})$ such that
	\begin{eqnarray}\label{eq5.9}
		\lim\limits_{k\in K,~k\rightarrow\infty}P_{T_{S}(x^{k})}(-\bar{u}^{k})=0.
	\end{eqnarray}
	
	Using \eqref{eq5.5}, \eqref{eq5.8}, and the properties of the projected gradient ([\cite{calamai1987projected}, Lemma 3.1]), we immediately get that
	\begin{eqnarray*}
		\alpha&\leq &\langle\bar{v}^{k},g_{k}\rangle\\
		&=&\langle-\bar{u}^{k},-g_{k}\rangle-\langle\bar{u}^{k}-\bar{v}^{k},g_{k}\rangle\\
		&\leq &\max\{\langle-\bar{u}^{k},d\rangle\mid d\in T_{S}(x^{k}),\|d\|\leq 1\}-\langle\bar{u}^{k}-\bar{v}^{k},g_{k}\rangle\\
		&=&\|P_{T_{S}(x^{k})}(-\bar{u}^{k})\|-\langle\bar{u}^{k}-\bar{v}^{k},g_{k}\rangle.
	\end{eqnarray*}
	According to \eqref{eq5.3} and \eqref{eq5.9},
	\begin{eqnarray*}
		\alpha&\leq &\liminf_{k\in K,~k\rightarrow\infty}\{\|P_{T_{S}(x^{k})}(-\bar{u}^{k})\|-\langle\bar{u}^{k}-\bar{v}^{k},g_{k}\rangle\}\\
		&=&-\limsup\limits_{k\in
			K,~k\rightarrow\infty}\langle\bar{u}^{k}-\bar{v}^{k},g_{k}\rangle\leq
		0,
	\end{eqnarray*}
	which leads to a contradiction. The proof is complete.
\end{proof}

Applying Theorem \ref{Theorem 5.1} to the special cases (Examples \ref{Example 2.1}-\ref{Example 2.4}) of $EP(\phi, S)$, we can obtain the following corollaries.

\begin{corollary}\label{Corollary 5.1}
	The following conclusions hold.
	
	\hspace{-8mm}(1) In the $(MP)$, suppose that $\bar{S}\subset S$ is a closed
	set, $\partial f(x)\neq\emptyset$ for $\forall x\in S$, and
	$\bar{S}$ is augmented weak sharp with respect to
	$\{x^{k}\}\subset S$. So the following conclusions are
	established.
	
	(i)Suppose that \eqref{eq1.2} holds. If $\{x^{k}\}$ terminates finitely to
	$\bar{S}$, then we have
	\begin{eqnarray}\label{eq5.10}
		0\in\liminf_{k\rightarrow\infty } P_{T_{S}(x^{k})}(-\partial
		f(x^{k})).
	\end{eqnarray}
	
	(ii) If \eqref{eq5.10} holds, then $\{x^{k}\}$ terminates finitely to
	$\bar{S}$.
	
	\hspace{-8mm}(2) In the $VIP(F,S)$, suppose that $\bar{S}\subset S$ is a closed
	set, $\bar{S}$ is augmented weak sharp with respect to
	$\{x^{k}\}\subset S$. Then $\{x^{k}\}$ terminates finitely to
	$\bar{S}$ if and only if
	\begin{eqnarray}\label{eq5.11}
		\lim\limits_{k\rightarrow\infty}P_{T_{S}(x^{k})}(-F(x^{k}))=0.
	\end{eqnarray}
	
	\hspace{-8mm}(3) In the $(SPP)$, suppose that $\bar{S}\subset S$ is a closed set,
	for $\forall x\in S, \partial\varphi(x)\neq\emptyset$, $\bar{S}$ is
	augmented weakly sharp with respect to ${\{x^{k}\}}\subset S$. So
	the following conclusions are established.
	
	(i) Suppose that \eqref{eq1.2} holds. If $\{x^{k}\}$ terminates finitely
	to $\bar{S}$, then we have
	\begin{eqnarray}\label{eq5.12}
		0\in\liminf_{k\rightarrow\infty
		}P_{T_{S}(x^{k})}((-\partial_{y_{1}}\varphi(x^{k}),\partial_{y_{2}}\varphi(x^{k}))).
	\end{eqnarray}
	
	(ii) If \eqref{eq5.12} holds, then $\{x^{k}\}$ terminates finitely to
	$\bar{S}$.
	
	\hspace{-8mm}(4) In the $(NEP)$, suppose that $\bar{S}\subset S$ is a closed
	set, $\partial_{y_{i}}f_{i}(x)\neq\emptyset$ for $\forall x\in S$
	and $i\in I={\{1,2,..,n\}}$, $\bar{S}$ is augmented weak sharp
	with respect to $\{x^{k}\}\subset S$. So the following conclusions
	are established.
	
	(i) Suppose that \eqref{eq1.2} holds. If $\{x^{k}\}$ terminates finitely
	to $\bar{S}$, then we have
	\begin{eqnarray}\label{eq5.13}
		0\in\liminf_{k\rightarrow\infty} P_{T_{S}(x^{k})}(-(\partial_{y_{i}}f_{i}(x^{k}),i\in I)).
	\end{eqnarray}
	
	(ii) If \eqref{eq5.13} holds, then $\{x^{k}\}$ terminates finitely to
	$\bar{S}$.
\end{corollary}

Next, we apply Theorem \ref{Theorem 5.1} to a kind of important function, i.e., for $\forall x\in S$, $\phi(x,\cdot)$ is locally Lipschitzian function
on $\mathbb{R}^{n}$. For this function, by [\cite{rockafellar1998variational}, Theorem 9.13 and Theorem 8.15],
we know \eqref{eq1.2} is established, and $\partial_y\phi(x,x)\neq\emptyset$. Therefore by Theorem \ref{Theorem 5.1}, we have the following corollary.
\begin{corollary}\label{Corollary 5.2}
	In the $EP(\phi,S)$, suppose that $\bar{S}\subset S$ is a closed set, for $\forall x\in S$, $\phi(x,\cdot)$ is locally Lipschitzian, $\bar{S}$ is augmented weak sharp with respect to ${\{x^{k}\}}\subset S$, then $\{x^{k}\}\subset S$ terminates finitely to $\bar{S}$, if and only if \eqref{eq5.1} holds.
\end{corollary}

By Corollary \ref{Corollary 5.2} and Proposition \ref{Proposition 3.1}, we immediately get the following corollary.
\begin{corollary}\label{Corollary 5.3}
	In the $EP(\phi,S)$, suppose $\bar{S}\subset S$ is a closed set, $\phi(x,\cdot)$ is locally Lipschitzian function on $\mathbb{R}^{n}$ for $\forall x\in S$, and $\partial_y\phi(x,x)$ is monotone over $S$. If $\bar{S}$ is a weak sharp set, then $\{x^{k}\}\subset S$ terminates finitely to $\bar{S}$, if and only if \eqref{eq5.1} holds.
\end{corollary}

\begin{remark}\label{Remark 5.1}
	By Remark \ref{Remark 3.2}, one can see that the monotonicity of $\partial_{y}\phi(x,x)$ does not imply the convexity of $\phi(x,\cdot)$, and the reverse is also true. For example, $\phi(x,y)=x^{2}y-x^{3},(x,y)\in R\times R$.
\end{remark}

Notice that a finite convex function on $\mathbb{R}^{n}$ is locally Lipschitzian, therefore by Corollary \ref{Corollary 5.3} and Proposition \ref{Proposition 3.2}, we immediately get the following corollary.

\begin{corollary}\label{Corollary 5.4}
	In the $EP(\phi,S)$, suppose $\bar{S}\subset S$ is a closed set, for $\forall x\in S$, $\phi(x,\cdot)$ is convex function on $\mathbb{R}^{n}$, and $\phi(\cdot,\cdot)$ is monotonic over $S\times S$. If $\bar{S}$ is a weak sharp set, then $\{x^{k}\}\subset S$ terminates finitely to $\bar{S}$, if and only if \eqref{eq5.1} holds.
\end{corollary}

For the special cases (Examples \ref{Example 2.1}-\ref{Example 2.3}) of $EP(\phi,S)$, we have the following corollaries.

\begin{corollary}\label{Corollary 5.5}
	The following conclusions are established.
	
	(1) In the convex programming $(MP)$, suppose $\bar{S}\subset S$
	is a weak sharp minimal set. Then $\{x^{k}\}\subset S$ terminates
	finitely to $\bar{S}$, if and only if \eqref{eq5.10} holds.
	
	(2) In the $VIP(F,S)$, suppose $F(\cdot)$ is monotonic over $S$,
	$\bar{S}\subset S$ is a closed and weak sharp set. Then
	$\{x^{k}\}\subset S$ terminates finitely to $\bar{S}$, if and only
	if \eqref{eq5.11} holds.
	
	(3) In the $(SPP)$, for $\forall x_{2}\in S_{2}$,
	$\varphi(\cdot,x_{2})$ is a convex function over $\mathbb{R}^{n_1}$; for
	$\forall x_{1}\in S_{1}$, $\varphi(x_{1},\cdot)$ is a concave
	function over $\mathbb{R}^{n_2}$, and $\bar{S}\subset S$ is a weak sharp
	set. Then $\{x^{k}\}\subset S$ terminates finitely to $\bar{S}$, if
	and only if \eqref{eq5.12} holds.
\end{corollary}
\begin{proof}
	According to the hypotheses in (1) and (2), one can see that they all meet the hypotheses conditions of Corollary \ref{Corollary 5.3}. For (3), by its hypotheses and \eqref{eq2.3}, one can see that the hypotheses conditions of Corollary \ref{Corollary 5.4} are satisfied. So we immediately get that these conclusions (1)-(3) hold.
\end{proof}

\begin{remark}\label{Remark 5.2}
	Corollary \ref{Corollary 5.5} (1) is equivalent to [\cite{zhou2012new}, Theorem 3.1].
\end{remark}

By Corollary \ref{Corollary 5.2} and Proposition \ref{Proposition 3.3}, we can obtain the following corollary.

\begin{corollary}\label{Corollary 5.6}
	Under the hypothesis in Theorem \ref{Theorem 3.1}, and suppose $\bar{S}\subset S$ is closed and $\partial_{y}\phi(x,x)$ is monotonic over $S$. If for $\forall x\in\bar{S}$,
	$$-\partial_{y}\phi(x,x)\cap(-N_{S}(x))\subset \text{int} G ,$$
	then $\{x^{k}\}\subset S$ terminates finitely to $\bar{S}$, if and only if \eqref{eq5.1} holds.
\end{corollary}
Since the convex programming $(MP)$, as a special case of $EP(\phi,S)$, satisfies the hypotheses in Corollary \ref{Corollary 5.6}, we can obtain the following corollary.

\begin{corollary}\label{Corollary 5.7}
	Suppose that in the convex programming $(MP)$, it holds that for $\forall x\in\bar{S}$, 
	$$-\partial f(x)\cap(-N_{S}(x))\subset \text{int} G.$$
	Then $\{x^{k}\}\subset S$ terminates finitely to $\bar{S}$, if and only if \eqref{eq5.10} holds.
\end{corollary}

\begin{remark}\label{Remark 5.3} Corollary \ref{Corollary 5.7} is a generalization of [\cite{burke1993weak}, Theorem 4.7] in the smooth convex programming, and in Corollary \ref{Corollary 5.7}, the two hypotheses about $\{x_{k}\}$ and $\nabla f(\cdot)$ in [\cite{burke1993weak}, Theorem 4.7] are removed.
\end{remark}

Next, we consider the smooth situation, the case that $\phi(x,\cdot)$
is locally Lipschitzian. Therefore, by Proposition \ref{Proposition 3.4} and
Corollary \ref{Corollary 5.2}, we obtain the following corollary.
\begin{corollary}\label{Corollary 5.8}
	In the $EP(\phi,S)$, suppose $\bar{S}\subset
	S$ is a closed set, and $\{x^{k}\}\subset S$ satisfies \eqref{eq3.11}. If
	$\bar{S}$ is weak sharp, then $\{x^{k}\}$ terminates finitely to
	$\bar{S}$, if and only if
	\begin{eqnarray}\label{eq5.14}
		\lim\limits_{k\rightarrow\infty}P_{T_{S}(x^{k})}(-\nabla_{y}\phi(x^{k},x^{k}))=0.
	\end{eqnarray}
\end{corollary}

By Proposition \ref{Proposition 3.5} and Corollary \ref{Corollary 5.2}, we obtain the following corollary.
\begin{corollary}\label{Corollary 5.9}
	In the $EP(\phi,S)$, suppose $\bar{S}\subset S$ is a closed set, and $\{x^{k}\}\subset S$, $\{\nabla_{y}\phi(x^{k},x^{k})\}$ is bounded and any of its accumulation $\bar{p}$ satisfies $-\bar{p}\in \text{int} G$. Then
	$\{x^{k}\}$ terminates finitely to $\bar{S}$, if and only if \eqref{eq5.14} holds.
\end{corollary}

\begin{remark}\label{Remark 5.4}
	In the special cases $VIP(F,S)$ of $EP(\phi,S)$, Corollary \ref{Corollary 5.9} is a generalization and improvement of [\cite{zhou2012new}, Theorem 3.3], i.e., $[T_S(x)\cap N_{\bar{S}}(x)]^\circ$ in [\cite{zhou2012new}, Theorem 3.3] is replaced with $[T_S(x)\cap \hat{N}_{\bar{S}}(x)]^\circ$, and the assumption of the continuity of $F(\cdot)$ is removed. It is worthwhile to note that [\cite{zhou2012new}, Theorem 3.3] has ever improved [\cite{xiu2005finite}, Theorem 3.2].
\end{remark}

By Proposition \ref{Proposition 3.6} and Corollary \ref{Corollary 5.2}, we have the following corollary.
\begin{corollary}\label{Corollary 5.10}
	In the $EP(\phi,S)$, suppose $\nabla_{y}\phi(x,x)$ is continuous over $S$, $\{x^{k}\}\subset S$ is bounded and any of its accumulation is strongly non-degenerate. Then $\{x^{k}\}$ terminates finitely to $\bar{S}$, if and only if the \eqref{eq5.14} holds.
\end{corollary}

\begin{remark}\label{Remark 5.5}
	In the smooth programming problems $(MP)$, a special case of $EP(\phi,S)$, Corollary \ref{Corollary 5.10} is just [\cite{wang2013global}, Theorem 5.3], and the latter is an extension of [\cite{burke1988identification}, Corollary 3.5].
\end{remark}

By Theorem \ref{Theorem 5.1} and a series of its corollaries, one can see that, under normal conditions, the weak sharpness or strong non-degeneracy
of the solution set is a special case of the augmented weak sharpness
with respect to the feasible solution sequence. On the other hand,
for some algorithms in mathematical programming and variational
inequalities, for example, the proximal point algorithm, the
gradient projection algorithm and the $SQP$ algorithm and so on (see \cite{burke1988identification,calamai1987projected,polyak1987introduction,wang2005convergence,wang2013global,xiu2005finite,xiu2003some}), the projected gradient of the point sequence generated by them all converge to zero, i.e. \eqref{eq5.1} holds. Therefore, the notion of augmented weak sharpness of the solution set presented by us provides weaker sufficient conditions than the weak sharpness or strong non-degeneracy for the finite termination of these algorithms.

\section{Conclusion}
In this paper, a novel concept concerning the solution set of equilibrium problems has been introduced, namely, the augmented weak sharpness of the solution set. This concept extends the traditional notions of weak sharpness and strong non-degeneracy in relation to feasible solution sequences. We have established that the necessary and sufficient conditions for the finite termination of feasible solution sequences are met when the lower limit of the projected sub-differential sequence, associated with the feasible solution sequence, encompasses the zero point. This finding has led to the derivation of several significant corollaries. Additionally, the augmented weak sharpness of the solution set presents a sufficient condition for the finite termination of certain equilibrium problem algorithms and their specific variants, such as mathematical programming problems, variational inequality problems, Nash equilibrium problems, and global saddle point problems. This condition is less stringent than the traditional criteria of weak sharpness and strong non-degeneracy.

	\bibliographystyle{siamplain}
	\bibliography{references}

\end{document}